\documentclass{mcom-l}

\usepackage{tikzuml}
\usepackage{amssymb}
\usepackage{graphicx}
\usepackage{listings}
\usepackage{algorithm}
\usepackage{algorithmic}
\usepackage[colorlinks]{hyperref}
\usepackage[T1]{fontenc}

\newtheorem{theorem}{Theorem}[section]
\newtheorem{lemma}[theorem]{Lemma}
\theoremstyle{definition}

\theoremstyle{remark}
\newtheorem{remark}[theorem]{Remark}

\numberwithin{equation}{section}
\numberwithin{algorithm}{section}

\begin{document}

\title{On the tame kernels of imaginary cyclic quartic fields with class number one}


\author{Zhang Long}
\address{School of Mathematics and Statistics, Qingdao University, Qingdao 266071, P.R. China;
Institute of Applied Mathematics of Shandong, Qingdao University,  Qingdao 266071, P.R. China}
\curraddr{}
\email{zhanglong\_note@hotmail.com}
\thanks{}

\author{Xu Kejian}
\address{School of Mathematics and Statistics, Qingdao University, Qingdao 266071, P.R. China;
Institute of Applied Mathematics of Shandong, Qingdao University,  Qingdao 266071, P.R. China}
\curraddr{}
\email{kejianxu@amss.ac.cn}
\thanks{}


\date{}

\dedicatory{}
\begin{abstract}
 Tate  first proposed a method to  determine $K_2\mathcal{O}_F,$ the tame kernel of $F,$ and    gave the concrete computations for some special quadratic fields with small discriminant. After that,  many examples for quadratic fields with larger discriminants are given,
 and similar works also have been done for cubic  fields and for  some special quartic fields with discriminants not large.
 
 In the present paper, we investigate the case of  more general imaginary cyclic quartic field  $F=\mathbb{Q}\Big(\sqrt{-(D+B\sqrt{D})}\Big)$ with  class number one and large discriminants.
 The key problem is how to decrease the huge theoretical bound appearing in the computation to a manageable one and  the main difficulty   is  how to deal with  the large-scale data emerged in the process of computation.
 To solve this problem we have established a general architecture for the computation,  in particular we have done the works:
 (1) the PARI's functions are invoked in C++ codes;
 (2) the parallel programming approach is used in C++ codes;
 (3) in the design of algorithms and codes, the object-oriented viewpoint is used, so an extensible program is obtained.
 
 As an application of our program, we prove that $K_2\mathcal{O}_F$ is trivial
 in the following three cases: $B=1,D=2$ or $B=2, D=13$ or $B=2, D=29.$
 In the last case, the discriminant of $F$ is 24389, hence, we can claim   that our architecture  also works for the  computation of  the tame kernel of a number field with discriminant less than 25000.

\end{abstract}

\maketitle


\begin{section}{Introduction}

Let $F$ be a number field and $\mathcal{O}_{F}$ the ring of algebraic integers of $F,$ and let $K_{2}\mathcal{O}_{F}$ denote the $K_2$ of $\mathcal{O}_{F}.$
Garland \cite{Garland001} proved that $K_{2}\mathcal{O}_{F}$ is a finite abelian group. However, $K_{2}\mathcal{O}_{F}$ can be regarded as tame kernel.

In fact, let $K_2F$ be the Milnor $K_2$-group, and let  $k_v=\mathcal{O}_{F}/\mathcal{P}_v$
and $k^{*}_v$ the multiplicative group of $k_v$, where  $\mathcal{P}_v$ is the prime ideal corresponding to a finite prime place $v.$
Then we have the well-known tame homomorphism:
\begin{equation*}
\begin{split}
\partial_{v}:K_{2}F\rightarrow{k^{*}_v}
\end{split}
\end{equation*}
which is 	defined by
\begin{equation*}
\begin{split}
\partial_{v}(\{x,y\})={(-1)^{v(x)v(y)}\frac{x^{v(y)}}{y^{v(x)}}(\mbox{mod}\, \mathcal{P}_{v})},
\end{split}
\end{equation*}
where $v(x), v(y)$ denote the valuations  of $x,y$ with respect to  the prime $v$ respectively,  and thus we have
$$\partial=\bigoplus_{v}\partial_{v}:K_{2}F\rightarrow{\bigoplus_{v}k^{*}_{v}},$$
where $v$ runs over all finite places. The kernel  ker$\partial$  is called  the tame kernel of the field $F.$ D.Quillen \cite{Quillen}proved that ker$\partial=K_2\mathcal{O}_F.$

There is no an effective algorithm for determining the tame kernel of a given number field directly, because it is defined noneffectively,
The first   method of determining   the tame kernel of a given number field was proposed by J.Tate \cite{Tate001}. Now, we describe Tate's method in more  details.

Let  $Nv$  be the number $|k_v|,$ which is called the norm of $v$, and
let $v_{1},v_{2},v_{3}, \ldots,v_n,\ldots$
be all finite places of $F$ ordered in such a way that $Nv_{i }\leq{Nv_{i+1}}$, for $i=1,2,3,\cdots.$

Let $  S_{m}=\{v_{1},\cdots,v_{m}\}$ ($S_{0}= \emptyset$), and let
$$ \mathcal{O}_{m}=\{a\in{F}: v(a)\geq0,v\not\in{S_{m}}\},$$
\begin{equation*}
\begin{split}
U_{m}=\{a\in{F}: v(a)=0,v\not\in{S_{m}}\}.
\end{split}
\end{equation*}
Thus $ \mathcal{O}_{0}$ and $U_{0}$ are just  the ring of algebraic integers
and  the group of units  respectively.
\par Let $K_{2}^{S_m}F$ be the subgroup of $K_{2}F$ generated by symbols $\{x,y\}$, where $x,y\in{U_{m}}$. Then we have $K_{2}F=\bigcup_{m=1}^{\infty}K_{2}^{S_m}F.$
Clearly,  $\partial_{v_{m}}$
induces the  homomorphism
\begin{equation*}
\begin{split}
\partial_{v_{m}}:\frac{K_{2}^{S_m}F}{K_{2}^{S_{m-1}}F} \longrightarrow {k^{*}_{v_{m}}}.
\end{split}
\end{equation*}
Bass and Tate \cite{Bass001} proved that for  sufficiently large $m,$  $\partial_{v_{m}}$ is  isomorphic,
which implies
\begin{equation*}
\begin{split}
K_{2}\mathcal{O}_{F}=\mbox{ker}\Big(\partial: K_2^{S_m}F\longrightarrow \coprod_{v\in S_m}k^*_v\Big).
\end{split}
\end{equation*}

Thus, if we can make the large $m$ as small as possible and get sufficiently many
relations satisfied by elements of $K_{2}^{S_{m}}F$, then we may
determine the tame kernel $K_{2}\mathcal{O}_{F}.$  So the problem is reduced to finding conditions for $\partial_{v_{m}}$ to be  isomorphic for  sufficiently large $m.$  The conditions  were found   by Tate.

Assume that  the prime ideal $\mathcal{P}_{m}$ of $ \mathcal{O}_{m}$
corresponding to $v_{m}$ is generated by $\pi_{m}.$ Define the morphisms:
$$\ \ \alpha: U_m\longrightarrow   \frac{K_{2}^{S_{m}}F}{K_{2}^{S_{m-1}}F},  \ \ \ \   \alpha(u)=\{u,\pi_{m}\}(\mbox{mod}\, K_{2}^{S_{m-1}}F),$$
$$\beta: U_m\longrightarrow   k^{*}_{v_{m}},  \ \ \ \   \beta(u)=u(\mbox{mod}\,\pi_{m}.\ \ \ \ \ \ \ \ \ \ \ \ \ \ \ \ \ \ \ \ \ $$
Then the conditions found by Tate are  presented in the following theorem.

\begin{theorem}\cite{Tate001}
	Suppose that prime ideal   $\mathcal{P}$ corresponding to a finite place $v\not\in{S_{m}}$   is generated by $\pi\in{ \mathcal{O}_{F}}$ and that $U_{1}'$ is a group generated by $(1+\pi{U_{m}})\bigcap{U_{m}}$.
	If
	there are  subsets $W_{m},C_{m},G_{m}$ of $U_{m}$ satisfying the following conditions:
	
	(i) $W_{m}\subseteq{C_{m}U_{1}'}$ and $U_{m}$ is generated by $W_{m}$,
	
	(ii) $C_{m}G_{m}\subseteq{C_{m}U_{1}'}$ and $k^{*}_v$ is generated by $\beta{(G_{m})}$,
	
	(iii) $1\in{C_{m}\bigcap{ker\beta}}\subseteq{U_1'},$\\
	then $\partial{_{v}} $ is an isomorphism.
\end{theorem}

Hence, according to Tate's above method, to determine the tame kernel of a given number field, it suffices to construct suitable subsets $W_{m},C_{m},G_{m}$ of $U_{m}$  and determine the bound of  $m.$

Using his method, Tate could give  the analysis for the six first imaginary quadratic cases because in these cases the bound of $m$ is very small. More precisely, let $F=\mathbb{Q}(\sqrt{-d}).$ Then Tate proved that $K_2\mathcal{O}_F$ is trivial if $d=1,2,3,11,$ and $K_2\mathcal{O}_F\cong \mathbb{Z}/2\mathbb{Z}$ if $d=7,15.$

Subsequently, Qin \cite{Qin001,Qin002} investigated the cases  $d=6$ and $35$  with a modification of the choice of the subset $C_m$ in Tate's method, and   nearly at the same time, Ska\l ba \cite{Skalba001} gave the computations of the cases $d=5$ and $19$ with the help of his generalized Thue theorem (GTT); essentially, it is also a modification of the choice of $C_m.$
After that, for quadratic fields Browkin improved Ska\l ba's method to get a more accurate
bound of $m,$  which allowed him to compute the cases $d=23$ and $31$ \cite{Browkin001,Browkin002}.
It should be pointed out that  all of these works  were done by hand.

The further computations for quadratic fields are due to Belabas and Gangl who used computers and determined the tame kernel for   all $d$ up to  $1000$ with only $7$ exceptions. \cite{Belabas001}

The tame kernels of cubic  fields had been investigated by Browkin in \cite{Browkin004}. His numerical computations were performed using the package PARI/GP.

The cases of quartic fields are more complicated. Using Ska\l ba's GTT,
Guo proved that $K_2\mathcal{O}_F$ is trivial
when $F=\mathbb{Q}(\zeta_8)$ (see \cite{XuejunGuo001}). He did it also by hand.
When $F=\mathbb{Q}(\zeta_5)=\mathbb{Q}(\sqrt{-(5+2\sqrt{5})}),$ under the assumption of the Lichtenbaum conjecture,
Browkin once conjectured in the  paper(\cite{Browkin003} that the tame kernel $K_2\mathcal{O}_F$ is trivial.
In a recent paper, we confirmed Browkin's conjecture \cite{zx}. However, the arithmetic properties of   field $\mathbb{Q}(\zeta_5)$ are much more complicated than those of  quadratic fields and biquadratic fields.
Therefore the discussion is longer, and more cases are  considered. Actually,
we have  to use   PARI/GP and some other algorithms.

For  further computations, the bound $m$ should be determined theoretically. This was solved by R.Groenewegen \cite{GROENEWEGEN001}, who gave a  theoretical bound of $m$ for a general number field. In this paper, for the cyclic quartic field we also find a way to obtain the theoretical bound, and  in some cases our bounds are better than Groenewegen's (Remark 4.7).

Thus, for a given number field, if the theoretical bound is good enough, that is, if it is   a manageable, in another words, if the computation can be done by hand, then through constructing enough relations, we may determine the tame kernel of the given number field.
But unfortunately, actually these theoretical bounds may be very large,  far from being manageable.
This weak point makes the concrete computation nearly impossible for a higher degree number field, even  for a cyclic quartic field.

Hence, a new problem arises:

{\bf Problem:} {\it Whether one can give a  practical  method to decrease the theoretical bound to a manageable one ?}

Belabas and Gangl \cite{Belabas001} considered this problem. In order to get a manageable bound of $m,$  they proceed as follows.
Let $T=S\cup \{v\}$ and assume that $K_2\mathcal{O}_F\subseteq K_2^{T}F.$ They want to prove that, in fact, we already have
$K_2\mathcal{O}_F\subseteq K_2^{S}F.$ This will be used in the following situation: starting from the  initial $S$ determined by the theoretical bound, we iterate this process, successively truncating $S$ by deleting its last element with respect to the given ordering, hoping to reduce the set of places to a manageable size.
This   is a very natural way to decrease the theoretical bound to a manageable one, which has been used by many authors.
But again unfortunately its concrete realization  is   not easy in general.

In fact, if the discriminant of the given number field is not large,  then the difference between the theoretical bound and the manageable one is not large either, so it is  easy to for  one to do the work by writing a simple program or computing manually;
but, if the discriminant is large, then the difference between the theoretical bound and the manageable one is also large, so the work-load must increase  exponentially, as Balabas told us in a private letter, hence, in this case, we must face some challenges coming from dealing with the computation-intense task in the process solving the complex question.

In order to realize their plan, in particular
in the  construction of the set $C,$  which is one of the most difficulties to overcome, Balabas and Gangl \cite{Belabas001}, use the following three algorithms:
a) Fincke and Pohst's algorithm;
b) Method of lattice;
c) LLL-algorithm.
Balabas and Gangl's plan was eventually adapted for arbitrary number fields and implemented in the PARI/GP scripting language, but so far, as they pointed out \cite{Belabas001}, parts of the program remain specific to the imaginary quadratic case.

In the present  paper, we give a completely different and new approach. The key idea is that we use Object-Oriented Programming(OOP) and the Multi-threaded Parallel Technology.

It is well known that the idea of Object-Oriented Programming(OOP), developed as the dominant programming methodology in the early and mid 1990s, is to design data forms that correspond to the essential features of a problem.
So OOP brings a new approach to the challenge of the large-scale programming.\cite{prata}

In this  paper, to compute the tame kernels of imaginary cyclic quartic fields of class number one with large discriminant,
with the object-oriented viewpoint we   develop  a program,  the software framework of which is extensible and reusable and can be made as a base on which more tame kernels of number fields can be computed.
Moreover, to visualize the program's architectural blueprint, we also use Unified Modeling Language(UML)\cite{JCO}, which is a general-purpose, developmental and modeling language in the field of software engineering.

More precisely, in order to establish the software framework and visualize the architectural blueprint, we need to do the following works.
Firstly, we need to reduce Tate's theorem to a software engineering version, so as to give a main use case of a user's interaction with the system. The use case is not only a beginning of building the software framework but also a main driving force. In order to  visualize the use case,  we give the use case diagram (see Figure 1), which can be regarded as a UML description  of  Tate's theorem.
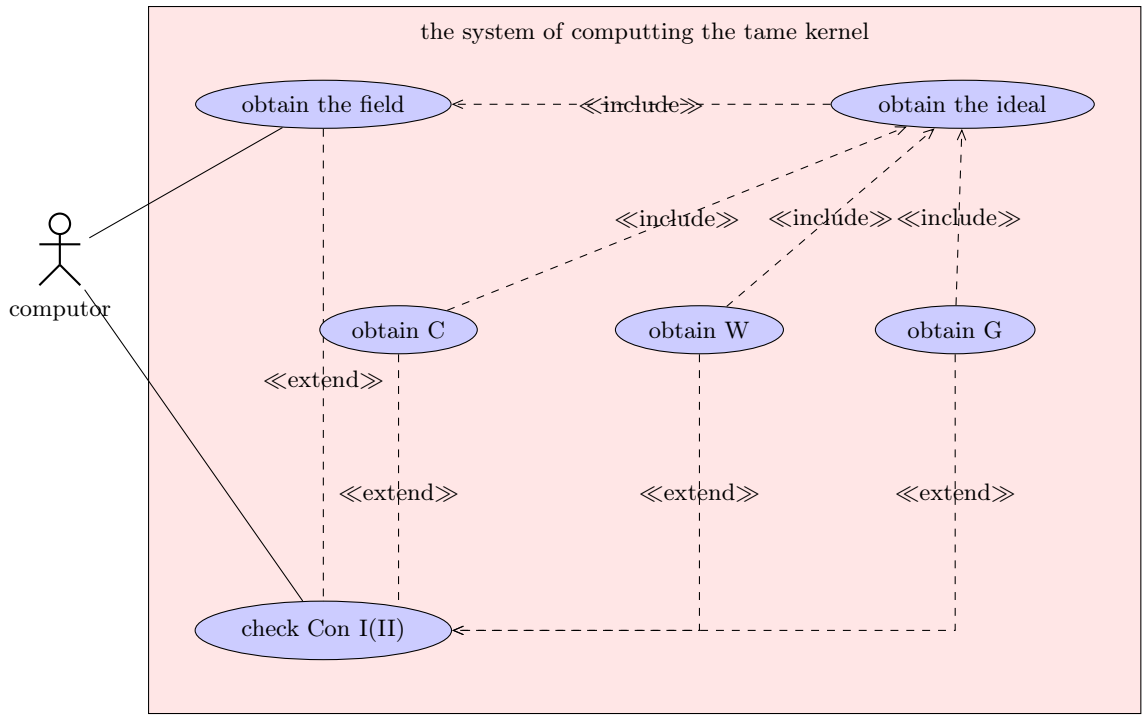
\begin{figure}[htbp]
	\begin{center}
		\begin{tikzpicture}
		\begin{umlsystem}[x=2.5, fill=red!10]{the system of computting the tame kernel}
		\umlusecase{obtain the field}
		\umlusecase[x=1,y=-3]{obtain C}
		\umlusecase[x=5.0,y=-3]{obtain W}
		\umlusecase[x=8.4,y=-3]{obtain G}
		\umlusecase[x=8.5,y=0] {obtain the ideal}
		\umlusecase[x=0, y=-7]{check Con I(II)}
		\end{umlsystem}
		\umlactor[x=-1,y=-2]{computor}
		\umlassoc{computor}{usecase-1}
		\umlassoc{computor}{usecase-6}
		\umlinclude [name=incl]{usecase-2}{usecase-5}
		\umlinclude [name=incl]{usecase-3}{usecase-5}
		\umlinclude [name=incl]{usecase-4}{usecase-5}
		\umlinclude [name=incl]{usecase-5}{usecase-1}
		\umlVHextend{usecase-2}{usecase-6}
		\umlVHextend{usecase-3}{usecase-6}
		\umlVHextend{usecase-4}{usecase-6}
		\umlVHextend{usecase-1}{usecase-6}
		\end{tikzpicture}
	\end{center}
	\caption{the use case diagram}
\end{figure}

Secondly,
we design  three classes: {\it CquarField}, {\it Cideal} and {\it Ccheck} as the structure of our program since in the view of OOP  all classes of a software constitute the core of the software framework. Moreover, using UML we  represent the relationships with the classes as a static class diagram, which
is generally used for general conceptual modelling of the systematic of a program
and for detailed modelling translating the models into programming code. The relationships with the three classes are represented the following static class diagram (see Figure 2) and the detailed design of the three classes and the static class diagram are introduced in (4.3).
\begin{figure}[htbp]
	\begin{center}
		\begin{tikzpicture}
		\begin{umlpackage}{p}[x=1,y=6]
		\umlclass[x=1,y=5,width=8ex]{Ccheck}{qf:CquarField}{{Ccheck(a:int, b:int, c:int, d:int)}
			\\ Pr\_chCndOn(n\_thrd:int):bool
			\\ Pr\_chCndTw(n\_thrd:int):bool}
		\umlclass[x=5,y=1,width=8ex]{CquarField}{}
		{
			CquarField(a:int, b:int, c:int, d:int)
			\\getPrimeTable(num\_condition):int
		}
		\umlclass[x=7,y=5,width=8ex]{Cideal}{qf:CquarField}
		{getSetInitW():GEN\\
			getSetInitG():GEN\\
			Para\_getSetInitC():void}
		\end{umlpackage}
		\umlunicompo[geometry=-|, mult=*, pos=1.7]{CquarField}{Cideal}
		\umlaggreg[geometry=|-, mult=1, pos=0.8, angle1=30, angle2=60, loopsize=2cm]{Ccheck}{CquarField}
		\end{tikzpicture}
	\end{center}
	\caption{the static class diagram}
\end{figure}
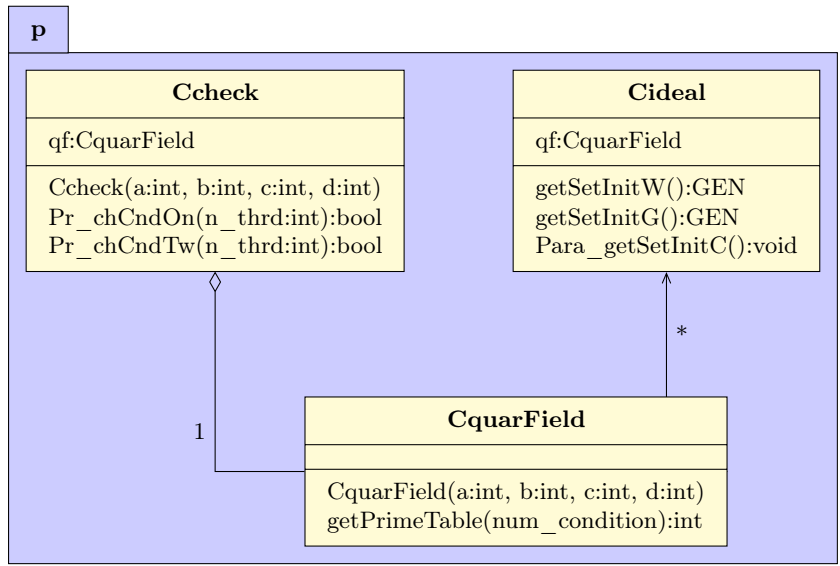

Finally,
in the view of software engineering, it is not enough to provide the use case diagram and the static class diagram to represent the program's  architecture, in another words, we must show  how objects operate with one another and in what order.
Since it is well known that in UML, a sequence diagram, which is an interaction diagram and also a construct of a message sequence chart, shows object interactions arranged in time sequence, thus we design the sequence diagram (see Figure 3) according to the relationships with the objects, which are represented in Tate's theorem, and in view of the difficulties  we must face during construction of the  program of computing the tame kernel of imaginary cyclic quartic field, such as large-scale computing.

\lstset{breaklines=true}
\hoffset=0cm
\voffset=-3.0cm	
\begin{figure}[htbp]
	\begin{center}
		\begin{tikzpicture}
		\begin{umlseqdiag}
		\umlactor[x=1]{computor}
		\umlobject[x=4]{checker}
		\umlobject[x=7.7]{qfCom}
		\umlobject[x=10]{idealCom}
		\umldatabase[x=13,class=DB, fill=blue!20]{db}
		\begin{umlcall}[op={check(a,b,c,d)}, type=synchron, return=1]{computor}{checker}
		\begin{umlcall}[op={CquarticField()}, type=asynchron, return=1]{checker}{qfCom}
		\end{umlcall}
		\begin{umlcall}[op={Pr\_getSetC()}, type=synchron, return=2]{computor}{checker}
		\begin{umlcall}[op={Part\_getSetC()}, type=asynchron, return=3]{checker}{checker}
		\begin{umlcall}[op={getPrimeTable(int)}, type=asynchron, return=3]{checker}{qfCom}
		\end{umlcall}
		\begin{umlcall}[op={pr\_getAllideal()}, type=synchron, return=3]{checker}{qfCom}
		\begin{umlcall}[op={pt\_getAlidl()}, type=asynchron, return=3]{qfCom}{qfCom}
		\end{umlcall}
		\end{umlcall}
		\begin{umlcall}[op={new ideal()}, type=asynchron, return=3]{checker}{idealCom}
		\end{umlcall}
		\begin{umlcall}[op={getSetInitG()}, type=asynchron, return=3]{checker}{idealCom}
		\end{umlcall}
		\begin{umlcall}[op={Pr\_getSetInitC()}, type=synchron, return=4]{checker}{idealCom}
		\begin{umlcall}[op={Pt\_getSetInitC()}, type=asynchron, return=3]{idealCom}{db}
		\end{umlcall}
		\end{umlcall}
		\end{umlcall}
		\end{umlcall}
		\begin{umlcall}[op={Pr\_chckCndOn()}, type=synchron, return=5]{computor}{checker}
		\begin{umlcall}[op={Part\_chCondOn(int)}, type=asynchron, return=5]{checker}{checker}
		\begin{umlcall}[op={ideal()}, type=asynchron, return=5]{checker}{idealCom}
		\end{umlcall}
		\begin{umlcall}[op={chCondOn()}, type=synchron, return=6]{checker}{idealCom}
		\begin{umlcall}[op={gp\_rd\_strm()}, type=asynchron, return=7]{idealCom}{db}
		\end{umlcall}
		\begin{umlcall}[op={getSetInitW()}, type=asynchron, return=5]{idealCom}{idealCom}
		\end{umlcall}
		\begin{umlcall}[op={getUm()}, type=asynchron, return=5]{idealCom}{idealCom}
		\end{umlcall}
		\end{umlcall}
		\end{umlcall}
		\end{umlcall}
		\begin{umlcall}[op={Pr\_chCndTw()}, type=synchron, return=8]{computor}{checker}
		\begin{umlcall}[op={Part\_chCndTwo()}, type=asynchron, return=5]{checker}{checker}
		\begin{umlcall}[op={ideal()}, type=asynchron, return=5]{checker}{idealCom}
		\end{umlcall}
		\begin{umlcall}[op={chCondTw()}, type=synchron, return=9]{checker}{idealCom}
		\begin{umlcall}[op={gp\_rd\_strm()}, type=asynchron, return=10]{idealCom}{db}
		\end{umlcall}
		\begin{umlcall}[op={getSetInitG()}, type=asynchron, return=5]{idealCom}{idealCom}
		\end{umlcall}
		\begin{umlcall}[op={getUm()}, type=asynchron, return=5]{idealCom}{idealCom}
		\end{umlcall}
		\end{umlcall}
		\end{umlcall}
		\end{umlcall}
		\end{umlcall}
		\end{umlseqdiag}
		\end{tikzpicture}
	\end{center}
	\caption{the sequence diagram}
\end{figure}
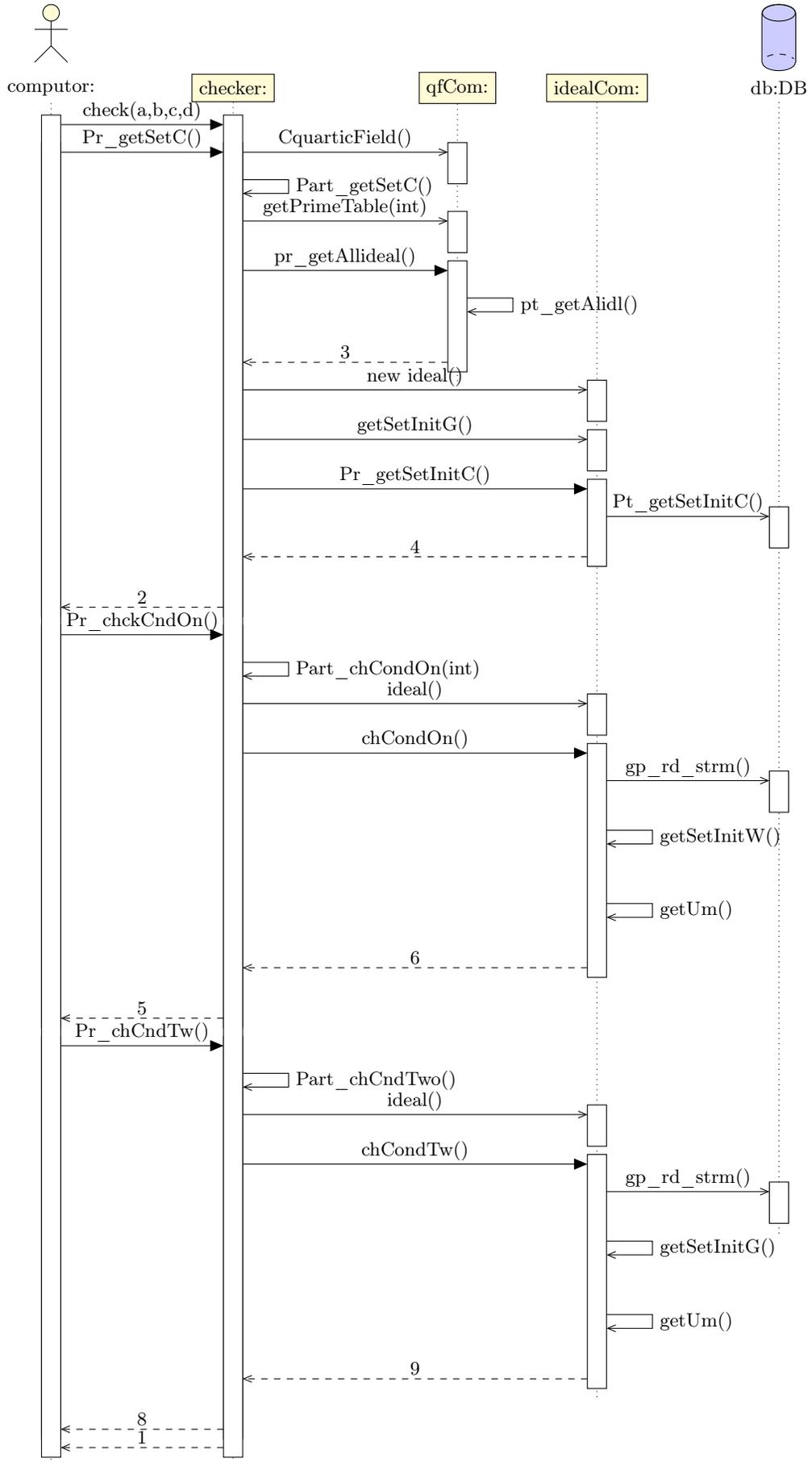

This is  what we have done in this paper in the  design of the program framework and  the program architecture in UML.
However, during building the program, we  meet   two difficulties.

One difficulty is how to create the codes which can be used to compute invariants of a number field.
Though some authors have designed some excellent algorithms for the computation,
the workload is so burdensome that it is almost impossible to implement so much algorithms for the  computation of tame kernels.
So it may be the  viable option to use the third-party libraries  to obtain the invariants.
Hence, PARI library, looked on as a reliable component, provides a powerful support to our program.

The other difficulty  is how to deal with the large-scale data emerging  in the process of computation.
In this study, we find that the  amount of  computation  of  tame kernels grows  explosively
as the discriminant and degree of extension of number fields get larger. In \cite{Belabas001} Belabas and Gangl
have computed some tame kernels of the quartic fields with
absolute values of  discriminants not large and
the workloads in computation of the tame kernels of these quartic fields are nearly equal to that of $F=\mathbb{Q}(\zeta_{5})$.
But now, as an example, we compute the tame kernel of $F=\mathbb{Q}(\sqrt{-(13+2\sqrt{13})})$ whose discriminant is 2917, and
we find that the workloads for  the computation of the tame kernels of  $F=\mathbb{Q}(\zeta_{5})$ and  $F=\mathbb{Q}(\sqrt{-(13+2\sqrt{13})})$
are not to be mentioned in the same breath.
In fact, in the case of  $F=\mathbb{Q}(\sqrt{-(13+2\sqrt{13})})$,  we once wrote some script codes with PARI/GP to compute its tame kernel.
After deploying the codes on PC and running about 24 hours, we make a rough estimate of running time.
It needs at least one year! So  these script codes are not time-base.

Thus, it is for this reason that motivates  us to design, in order to decrease the running time, the above  architecture, which is an extensible, reusable and component-based application
by associating  the Multi-threaded Parallel Technology and PARI library with the implemented architecture.
And at last,  deploying the application and running about 2 hours,
we  obtain the result of tame kernel of $F=\mathbb{Q}(\sqrt{-(13+2\sqrt{13})})$.

After that, we took about 3 months to compute the tame kernel of the number field $F=\mathbb{Q}(\sqrt{-(29+2\sqrt{29})})$ whose discriminant is $24389.$

In a private letter, Balabas told us that it took about 8 hours to obtain the tame kernel of $F=\mathbb{Q}(\sqrt{-(13+2\sqrt{13})})$ by a program  implementing the algorithms in the paper \cite{Belabas001}.
The program has been published in \url{https://www.math.u-bordeaux.fr/~kbelabas/research/software/K2-1.1.tgz.}
We also tried ever to use the same program to compute the tame kernel of
$F=\mathbb{Q}(\sqrt{-(29+2\sqrt{29})})$.
But, after  running the program about 2 hours, a bug  emerged and the program was interrupted.
This story implies that although some kind of problems can be solved efficiently by using the existing program without difficulty,
the computation of large-scale problems may be a nontrivial task, even a long-time running being acceptable, because of the restriction of the memory and CPU limitation.
Therefore, the design of programs as well as  its  efficiency  and reasonability may be essentially   depended on the scale of  computation.

Hence, as an application of our program,  now we are sure  from the above computation that our architecture  also works for the  computation of  the tame kernel of a number field with discriminant less than 25000.

In particular, as concrete examples,
we have proved the following theorem.

\begin{theorem}
	Let $F=\mathbb{Q}\Big(\sqrt{-(D+B\sqrt{D})}\Big)$ be cyclic quartic field.
	Then the tame kernel $K_{2}\mathcal{O}_{F}$ is trivial in the following cases:
	
	(i)  \cite{zx}  $B=2, D=5,$ i.e., $F=\mathbb{Q}(\zeta_5);$
	
	(ii) $B=1, D=2;$
	
	(iii)  $B=2, D=13;$
	
	(iv)  $B=2, D=29.$
\end{theorem}

\begin{remark}\quad (i) We have $h_F=1$  in the four cases in Lemma 2.1.
	
	(ii) By the present algorithms, the computation of the tame kernel of $\mathbb{Q}(\zeta_5)$ is quite easy.
\end{remark}
In the following, the conditions $W_{m}\subseteq{C_{m}U_{1}'}$ and $C_{m}G_{m}\subseteq{C_{m}U_{1}'}$ in Theorem 1.1  will be referred to be condition I and condition II
respectively.
\end{section}

\begin{section}{The cyclic quartic fields}
The following explit representation of a cyclic quartic field is proved in the reference \cite{HARDY001}.

\begin{lemma}\quad
	If $F$ is a real or imaginary cyclic quartic extensin of $\mathbb{Q}$, then there are integers A,B,C and D such that
	\begin{equation}
	F=\mathbb{Q}\Big(\sqrt{A(D+B\sqrt{D})}\Big)=\mathbb{Q}\Big(\sqrt{A(D-B\sqrt{D})}\Big)
	\end{equation}
	where
	\begin{equation}
	\begin{cases}
	A~ is~ squarefree~ and~ odd,\\
	D=B^2+C^2~is ~squarefree,~B>0,~C>0,\\
	A~and~D~are~relatively~prime.
	\end{cases}
	\end{equation}
\end{lemma}

Moreover, any field satisfying (2.1) and (2.2) is cyclic quartic extension of $\mathbb{Q},$ and the representation of $F$ is unique in the sense that if we have  another representation, say $F=\mathbb{Q}(\sqrt{A_1(D_1+B_1\sqrt{D_1})}),$ where $A_1, ~B_1,~C_1$ and $D_1$ are integers satisfying the conditions of (2), then $A=A_1, ~B=B_1,~C=C_1$and $D=D_1$.

On the other hand, it is given in the reference \cite{HARDY001}   a table of all the imaginary cyclic quartic fields
$F=\mathbb{Q}\Big(\sqrt{A(D+B\sqrt{D})}\Big)$, where $A, ~B,~C$ and $D$ are integers satisfying the condition (2.2). Now, we can list all imaginary cyclic quartic fields  with
class number one as follows.
\begin{equation*}
\begin{split}
\mbox{Case}\, 1:& ~F=\mathbb{Q}\Big(\sqrt{-(5+2\sqrt{5})}\Big), ~\mbox{where} ~A=-1, B=2, C=1, D=5;\\
\mbox{Case}\, 2:& ~F=\mathbb{Q}\Big(\sqrt{-(13+2\sqrt{13})}\Big), ~\mbox{where}~ A=-1, B=2, C=3, D=13;\\
\mbox{ Case}\, 3:& ~F=\mathbb{Q}\Big(\sqrt{-(2+\sqrt{2})}\Big), ~\mbox{where}~ A=-1, B=1, C=1, D=2;\\
\mbox{ Case}\, 4:& ~F=\mathbb{Q}\Big(\sqrt{-(29+2\sqrt{29})}\Big), ~\mbox{where}~ A=-1, B=2, C=5, D=29;\\
\mbox{ Case}\, 5:& ~F=\mathbb{Q}\Big(\sqrt{-(37+6\sqrt{37})}\Big), ~\mbox{where}~ A=-1, B=6, C=1, D=37;\\
\mbox{Case}\, 6:& ~F=\mathbb{Q}\Big(\sqrt{-(53+2\sqrt{53})}\Big), ~\mbox{where}~ A=-1, B=2, C=7, D=53;\\
\mbox{Case}\, 7:& ~F=\mathbb{Q}\Big(\sqrt{-(61+6\sqrt{61})}\Big), ~\mbox{where}~ A=-1, B=6, C=5, D=61;\\
\end{split}
\end{equation*}

In \cite{HUDSON001}, the integral basis of the cyclic quartic field $F=\mathbb{Q}\Big(\sqrt{A(D+B\sqrt{D})}\Big)$ is given as follows.
\begin{lemma}\quad
	Let $F=\mathbb{Q}\Big(\sqrt{A(D+B\sqrt{D})}\Big)$ be a cyclic quartic extension of $\mathbb{Q},$ where $A,B,C$ and $D$ satisfy the condition (2.2) in Lemma 2.1.  Set
	$$a'=\sqrt{A(D+B\sqrt{D})},\ \ ~b'=\sqrt{A(D-B\sqrt{D})}.$$
	Then an integral basis for $F$  is given as follows.
	\begin{equation*}
	\begin{split}
	(i)&\ \ \ \{1,\sqrt{D},a',b'\}~ if~ D\equiv{0}(mod\, 2);\\
	(ii)&\ \ \ \{1,\frac{1}{2}(1+\sqrt{D}),a',b'\}~
	if~ D\equiv{B}\equiv{1}(mod \, 2);\\
	(iii)&\ \  \{1,\frac{1}{2}(1+\sqrt{D}),\frac{1}{2}(a'+b'),\frac{1}{2}(a'-b')\}\\
	& if~ D\equiv{1}(mod \, 2),B\equiv{0}(mod \, 2),A+B\equiv{3}(mod \, 4);\\
	(iv)&\ \ \{1,\frac{1}{2}(1+\sqrt{D}),\frac{1}{4}(1+\sqrt{D}+a'+b'),\frac{1}{4}(1-\sqrt{D}+a'-b')\}\\
	& if~ D\equiv{1}(mod \, 2),B\equiv{0}(mod \, 2),A+B\equiv{1}(mod4),A\equiv{C}(mod \, 4);\\
	(v)&\ \ \{1,\frac{1}{2}(1+\sqrt{D}),\frac{1}{4}(1+\sqrt{D}+a'-b'),\frac{1}{4}(1-\sqrt{D}+a'+b')\}\\
	& if~ D\equiv{1}(mod \, 2),B\equiv{0}(mod \, 2),A+B\equiv{1}(mod \, 4),A\equiv{-C}(mod \, 4);\\
	\end{split}
	\end{equation*}
\end{lemma}

Hence,  the integral bases of Case 2,  of Case 1, Case 3, Case 4, Case 5, Case 7  and of Case 6 are respectively
$$\{1,\sqrt{D},a',b'\};$$
$$\{1,\frac{1}{2}(1+\sqrt{D}),\frac{1}{4}(1+\sqrt{D}+a'-b'),\frac{1}{4}(1-\sqrt{D}+a'+b')\};$$
$$\{1,\frac{1}{2}(1+\sqrt{D}),\frac{1}{4}(1+\sqrt{D}+a'+b'),\frac{1}{4}(1-\sqrt{D}+a'-b')\}.$$

\begin{lemma}\quad
	Let $F=\mathbb{Q}\Big(\sqrt{-(D+B\sqrt{D})}\Big)$ be a cyclic quartic extension of $\mathbb{Q}$ with class number $h(F)=1,$ where $B,C$ and $D$ satisfy the condition (2.2) in Lemma 2.1.
	Set $\beta=i\sqrt{D+B\sqrt{D}}$ and  $F=\mathbb{Q}(\beta).$ Then the following statements  hold.\\
	(i) The minimal polynomial  of $\beta$ over $\mathbb{Q}$ is
	$$f(x)=x^4+2Dx^2+(D^2-DB^2).$$
	(ii) The four conjugated roots of $\beta$ are
	$$\beta_1=\beta=ia,~\beta_2=\bar{\beta}=-ia,~\beta_3=ib,~\beta_4=-ib,$$
	where $a=\sqrt{D+B\sqrt{D}}$ and $b=\sqrt{D-B\sqrt{D}}.$\\
	(iii) The Galois group $Gal(F/\mathbb{Q})$ equals $\langle \sigma\rangle$ with $\sigma$ satisfying
	$$\sigma(\beta_1)=\beta_4,~\sigma(\beta_2)=\beta_3,~\sigma(\beta_3)=\beta_1,~\sigma(\beta_4)=\beta_2.$$
	(iv) The rank $r(U)$ of unit group $U$ of $F$ is $1.$ We denote the fundamental unit by $\xi.$\\
	(v)  In Case 1, Case 3, Case 4, Case 5, Case 7, the field $F$ has the same integral  base, which is
	$$\gamma_0=1,\gamma_1=\frac{1}{2}(1+\sqrt{D}),\gamma_2=\frac{1}{4}(1+\sqrt{D}+\beta-\beta_3),\gamma_3=\frac{1}{4}(1-\sqrt{D}+\beta+\beta_3).$$
	Moreover, the transition matrix  from $1,\beta,\beta^2,\beta^3$
	to $\gamma_0,\gamma_1,\gamma_2,\gamma_3$ is
	\begin{equation}
	M_1=
	\begin{pmatrix} &1 &0 &0 &0\\
	&\frac{B-D}{2B} &0 &-\frac{1}{2B} &0\\
	&\frac{B-D}{4B} &\frac{BC-B^2-D}{4BC} &-\frac{1}{4B} &-\frac{1}{4BC}\\
	&\frac{B+D}{4B} &\frac{BC+B^2+D}{4BC} &\frac{1}{4B} &\frac{1}{4BC}
	\end{pmatrix}
	\end{equation}
	(vi) In Case 2 and Case 6, the field $F$ has the same integral base, which is
	$$\gamma_0=1,\gamma_1=\frac{1}{2}(1+\sqrt{D}),\gamma_2=\frac{1}{4}(1+\sqrt{D}+\beta+\beta_3),\gamma_3=\frac{1}{4}(1-\sqrt{D}+\beta-\beta_3).$$
	Moreover, the transition matrix  from $1,\beta,\beta^2,\beta^3$
	to $\gamma_0,\gamma_1,\gamma_2,\gamma_3$ is
	\begin{equation}
	M_2=
	\begin{pmatrix} &1 &0 &0 &0\\
	&\frac{B-D}{2B} &0 &-\frac{1}{2B} &0\\
	&\frac{B-D}{4B} &\frac{BC+B^2+D}{4BC} &-\frac{1}{4B} &\frac{1}{4BC}\\
	&\frac{B+D}{4B} &\frac{BC-B^2-D}{4BC} &\frac{1}{4B} &-\frac{1}{4BC}
	\end{pmatrix}
	\end{equation}
\end{lemma}
\begin{proof}\quad
	The proofs of (i),(ii),(iii) and (iv) are easy. So we only prove (v) and (vi).

	We will express $\beta_3=ib$ by $1,\beta,\beta^2,\beta^3.$
	Assume that
	\begin{equation}
	\begin{split}
	\beta_3=ib=x_0+x_1ia+x_2(ia)^2+x_3(ia)^3,
	\end{split}
	\end{equation}
	where $a=\sqrt{D+B\sqrt{D}}$, $b=\sqrt{D-B\sqrt{D}}$ and $x_0,x_1,x_2,x_3 \in{\mathbb{Q}}.$ Then the following equations  hold:
	\begin{equation}
	\begin{split}
	x_0-a^2x_2=0
	\end{split}
	\end{equation}
	\begin{equation}
	\begin{split}
	b-ax_1+a^3x_3=0
	\end{split}
	\end{equation}	
	From (2.6), we have $x_0=x_2=0.$ However, from (2.7), we get
	\begin{equation*}
	\begin{split}
	b^2&=a^2(x_1-a^2x_3)^2\\
	&=a^2(x^2_1+a^4x_3^2-2a^2x_1x_3)\\
	&=(D+B\sqrt{D})[x_1^2+(D+B\sqrt{D})^2x_3^2-2(D+B\sqrt{D})x_1x_3]\\
	&=(D+B\sqrt{D})[x_1^2+(D^2+B^2D)x_3^2-2Dx_1x_3+(2BDx_3^2-2Bx_1x_3)\sqrt{D}]\\
	&=[Dx_1^2+(D^3+3B^2D^2)x_3^2-2(D^2+B^2D)x_1x_3]\\
	&+[Bx_1^2+(3BD^2+B^3D)x_3^2-4BDx_1x_3]\sqrt{D}.
	\end{split}
	\end{equation*}
	By comparing with the both sides of the equality, we can get the system of equations on $x_1$ and $x_3$
	\begin{equation}
	\begin{split}
	D=Dx_1^2+(D^3+3B^2D^2)x_3^2-2(D^2+B^2D)x_1x_3
	\end{split}
	\end{equation}
	\begin{equation}
	\begin{split}
	-B=Bx_1^2+(3BD^2+B^3D)x_3^2-4BDx_1x_3.
	\end{split}
	\end{equation}
	
	i.e.
	\begin{equation}
	\begin{split}
	1=x_1^2+(D^2+3B^2D)x_3^2-2(D+B^2)x_1x_3
	\end{split}
	\end{equation}
	\begin{equation}
	\begin{split}
	-1=x_1^2+(3D^2+B^2D)x_3^2-4Dx_1x_3.
	\end{split}
	\end{equation}
	Adding the two equations, we have
	\begin{equation}
	\begin{split}
	x_1^2+(2D^2+2B^2D)x_3^2-(3D+B^2)x_1x_3=0.
	\end{split}
	\end{equation}
	If $x_3=0$, clearly we have $x_1=0,$ impossible. Thus  $\frac{x_1}{x_3}$ is a root of the  equation:
	\begin{equation}
	\begin{split}
	x^2-(3D+B^2)x+2D(D+B^2)=0.
	\end{split}
	\end{equation}
	Clearly, $2D, D+B^2$ are the two roots of (2.13), so
	$$\frac{x_1}{x_3}=2D\ \ \mbox{or} \ \ \frac{x_1}{x_3}=D+B^2.$$
	
	If $x_1=2Dx_3,$ then from  (2.10), we can get that $(D^2-B^2D)x_3^2=1.$ So we have
	$$x_3=\pm\frac{\sqrt{D}}{CD},\ \ \ x_1=\pm\frac{2\sqrt{D}}{C}.$$
	However, putting these expressions in  (2.11), we get immediately a contradiction.
	Hence, we must have $x_1=(D+B^2)x_3$. Therefore from (2.11) we get
	$$-1=(B^2+D)^2x_3^2+(3D^2+B^2D)x_3^2-4D(D+B^2)x_3^2=-B^2C^2x_3^2.$$
	Thus $x_3=\pm\frac{1}{BC}.$
	We can check that $x_0=x_2=0,x_1=\frac{D+B^2}{BC}$ and $x_3=\frac{1}{BC}$ satisfy the equation (2.5), which means
	$$\beta_3=\frac{D+B^2}{BC}\cdot \beta+\frac{1}{BC}\cdot \beta^3.$$
	
	Note that  $\sqrt{D}=\frac{-D}{B}-\frac{\beta^2}{B}.$ Then, in   Case 1, Case 3, Case 4, Case 5, Case 7,  we can express $\gamma_0,\gamma_1,\gamma_2,\gamma_3$ by $1,\beta,\beta^2,\beta^3$ as follows.
	\begin{equation*}
	\begin{pmatrix} &\gamma_0\\&\gamma_1\\&\gamma_2\\&\gamma_3 \end{pmatrix}=
	\begin{pmatrix} 1 &0 &0 &0\\
	\frac{B-D}{2B} &0 &-\frac{1}{2B} &0\\
	\frac{B-D}{4B} &\frac{BC-B^2-D}{4BC} &-\frac{1}{4B} &-\frac{1}{4BC}\\
	\frac{B+D}{4B} &\frac{BC+B^2+D}{4BC} &\frac{1}{4B} &\frac{1}{4BC}
	\end{pmatrix}
	\begin{pmatrix} 1\\ \beta\\ \beta^2\\ \beta^3 \end{pmatrix}
	=M_1
	\begin{pmatrix} 1\\ \beta\\ \beta^2\\ \beta^3 \end{pmatrix}
	\end{equation*}
	Similarly, in Case 2 and Case 6,  we  get
	\begin{equation*}
	\begin{pmatrix} \gamma_0\\ \gamma_1\\ \gamma_2\\ \gamma_3 \end{pmatrix}=
	\begin{pmatrix} 1 &0 &0 &0\\
	\frac{B-D}{2B} & 0 &-\frac{1}{2B} &0\\
	\frac{B-D}{4B} &\frac{BC+B^2+D}{4BC} &-\frac{1}{4B} &\frac{1}{4BC}\\
	\frac{B+D}{4B} &\frac{BC-B^2-D}{4BC} &\frac{1}{4B} &-\frac{1}{4BC}
	\end{pmatrix}
	\begin{pmatrix} 1\\ \beta\\ \beta^2\\ \beta^3 \end{pmatrix}
	=M_2
	\begin{pmatrix} 1\\ \beta\\ \beta^2\\ \beta^3 \end{pmatrix}.
	\end{equation*}
\end{proof}

\end{section}

\begin{section}{The tame kernel of an imaginary cyclic  quartic field}

\begin{subsection}{Lemmas}
 \begin{lemma}
 	Let $F=\mathbb{Q}(\sqrt{-(D+B\sqrt{D}}))$ be a cyclic quartic field with the class number $h(F)=1$ and let $\beta=ia$ with $a=\sqrt{D+B\sqrt{D}}.$ Then, for any prime ideal $\mathcal{P}$ of $F$, there exists an element $\alpha\in{\mathcal{O}_F}$ satisfying\\
 	(i) $\mathcal{P}=(\alpha)$;\\
 	(ii) $|\sigma(\xi)|\le{\Big|\frac{\sigma(\alpha)}{\alpha}\Big|\le|\xi|},$ where $\xi$ is the foundanment unit of $F.$
 	Moreover, we have
 	$$\frac{|N(\alpha)|^{\frac{1}{4}}}{|\xi|^{\frac{1}{2}}}\le|\alpha|\le\frac{|N(\alpha)|^{\frac{1}{4}}}{|\sigma(\xi)|^{\frac{1}{2}}}.$$

 \end{lemma}
 \begin{proof}
Because the class number $h_F$ is $1,$ the prime ideal $\mathcal{P}$ of $F$ is a principal ideal, i.e. $\mathcal{P}=(y)$ for some $y\in{\mathcal{O}_{F}}.$\\

i). If $|\sigma(\xi)|\le\Big|\frac{\sigma(y)}{y}\Big|\le|\xi|,$
let $\alpha=y.$ Then the lemma is true.

ii). If $\Big|\frac{\sigma(y)}{y}\Big|>|\xi|,$
since $\Big|\frac{\sigma(\xi)}{\xi}\Big|<1,$ there is a positive integer $k$ satisfying
\begin{equation}
\begin{split}
\Big|\frac{\sigma(y)}{y}\Big|\Big|\frac{\sigma(\xi)}{\xi}\Big|^{k}\le|\xi|<\Big|\frac{\sigma(y)}{y}\Big|\Big|\frac{\sigma(\xi)}{\xi}\Big|^{k-1}.
\end{split}
\end{equation}
Let $\alpha=y\xi^k.$ Then, we get  $\Big|\frac{\sigma(\alpha)}{\alpha}\Big|\le|\xi|.$ However $\Big|\frac{\sigma(\alpha)}{\alpha}\Big|=|\sigma(\xi)|\Big|\frac{\sigma(y\xi^{k-1})}{y\xi^{k-1}}\Big|>|\sigma(\xi)||\xi|>|\sigma(\xi)|.$

iii). If $\Big|\frac{\sigma(y)}{y}\Big|<|\sigma(\xi)|,$
as  in ii), there is a positive integer $k$ such that
\begin{equation}
\begin{split}
 \Big|\frac{\sigma(y)}{y}\Big|\Big|\frac{\xi}{\sigma(\xi)}\Big|^{k-1}<|\sigma(\xi)|\leq \Big|\frac{\sigma(y)}{y}\Big|\Big|\frac{\xi}{\sigma(\xi)}\Big|^{k}.
\end{split}
\end{equation}
Let $\alpha=\frac{y}{\xi^k}.$ Thus by (14),we have
\begin{equation*}
\begin{split}
|\sigma(\xi)|\le\Big|\frac{\sigma(\alpha)}{\alpha}\Big|\le|\xi|.
\end{split}
\end{equation*}
So
\begin{equation*}
\begin{split}
|\sigma(\xi)||\alpha^2|\le|\sigma(\alpha)||\alpha|\le|\xi||\alpha|^2.
\end{split}
\end{equation*}
Hence
\begin{equation*}
\begin{split}
|\sigma(\xi)|^2|\alpha^4|\le|N(\alpha)|\le|\xi|^2|\alpha|^4.
\end{split}
\end{equation*}
Therefore
\begin{equation}
\begin{split}
\frac{|N(\alpha)|^{\frac{1}{4}}}{|\xi|^{\frac{1}{2}}}\le|\alpha|\le\frac{|N(\alpha)|^{\frac{1}{4}}}{|\sigma(\xi)|^{\frac{1}{2}}}.
\end{split}
\end{equation}
\end{proof}
We denote $[t]$ to the nearest integer number to $t$. Let $\{t\}=t-[t].$ So $\{t\}\in{[-\frac{1}{2},\frac{1}{2}]}.$
\begin{lemma}
	For any $0\neq \alpha, x\in{\mathcal{O}_F}$, there is a $y\in{\mathcal{O}_F}$ such that
$$|x-\alpha y|\le c_1|\alpha|,$$
$$|\sigma(x-\alpha y)|\le c_2|\sigma(\alpha)|,$$
where $c_{1},c_2$ are constants depending  only on the field $F,$ i.e., on $A,B,C$ and $D.$ So
 $$N(x-\alpha y)\le c_1^2c_2^2N(\alpha).$$
\end{lemma}
\begin{proof}
	Assume that $\frac{x}{\alpha}=k_0\gamma_0+k_1\gamma_1+k_2\gamma_2+k_3\gamma_3$ where $\gamma_0,\gamma_1,\gamma_2,\gamma_3$ are the integral basis of $F$ and $k_i\in{\mathbb{Q}},i=0,1,2,3.$ Let
$$y=[k_0]\gamma_0+[k_1]\gamma_1+[k_2]\gamma_2+[k_3]\gamma_3\in{\mathcal{O}_F}.$$
We will show that $y$ satisfies the requirement.

Suppose that
	\begin{equation*}
		\begin{split}
		z=&x-y\alpha=\Big(\sum_{i=0}^{3}{k_i\gamma_{i}}\Big)\alpha-\Big(\sum_{i=0}^{3}{[k_i]\gamma_{i}}\Big)\alpha=\Big(\sum_{i=0}^{3}{\{k_i\}\gamma_{i}}\Big)\alpha=
\Big(\sum_{i=0}^{3}{z_i\gamma_{i}}\Big)\alpha,
		\end{split}
	\end{equation*}
where $z_i=\{k_i\}\in[-\frac{1}{2},\frac{1}{2}]\cap \mathbb{Q}.$ 	

Let $M=M_1$ or $M_2,$ and let$z'=\sum_{i=0}^{3}{z_i\gamma_{i}}.$
We can compute the maximal value of $|z|.$
Let $g=|z|^2$. Then
\begin{equation*}
\begin{split}
g&=|z|^2=|z'|^2|\alpha|^2\\
 &=(z_0,z_1,z_2,z_3)M
 \begin{pmatrix}
 1 &-ia &-a^2 &ia^3\\
 ia &a^2 &-ia^3 &-a^4\\
 -a^2 &ia^3 &a^4 &-ia^5\\
 -ia^3 &-a^4 &ia^5 &a^6
 \end{pmatrix}
 M^{T}
 \begin{pmatrix}
 z_0\\
 z_1\\
 z_2\\
 z_3\\
 \end{pmatrix}|\alpha|^2\\
 &=(z_0,z_1,z_2,z_3)M
 H_1M^{T}
 \begin{pmatrix}
 z_0\\
 z_1\\
 z_2\\
 z_3
 \end{pmatrix}|\alpha|^2,
\end{split}
\end{equation*}
where
$$H_1:=\begin{pmatrix}
 1 &0 &-a^2 &0\\
 0 &a^2 &0 &-a^4\\
 -a^2 &0 &a^4 &0\\
 0 &-a^4 &0 &a^6
 \end{pmatrix}.$$

Let
\begin{equation*}
\begin{split}
 h_1(z_0,z_1,z_2,z_3):&=(z_0,z_1,z_2,z_3)M
 H_1M^{T}
 \begin{pmatrix}
 z_0\\
 z_1\\
 z_2\\
 z_3
 \end{pmatrix}.
\end{split}
\end{equation*}
By pari/gp, we can check that the values of $h_1(z_0,z_1,z_2,z_3)$ on those stationary
are zero. Thus $h_1(z_0,z_1,z_2,z_3)$ reaches its maximal value  on the boundary.

Hence, for any $A,B,C$ and $D,$ we have
$$|x-y\alpha|= |z|\le |z'||\alpha| \le c_{1}'^{\frac{1}{2}}|\alpha|,$$
where
$$c_{1}'=\mbox{max}\{h(z_0,z_1,z_2,z_3): z_i=-\frac{1}{2}~\mbox{or} ~\frac{1}{2},i=0,1,2,3\}.$$

Similarly, let $$h_2 (z_0,z_1,z_2,z_3)=(z_0,z_1,z_2,z_3)M
 H_2M^{T}
 \begin{pmatrix}
 z_0\\
 z_1\\
 z_2\\
 z_3
 \end{pmatrix}.$$
 with
 $$H_2:=\begin{pmatrix}
 1 &0 &-b^2 &0\\
 0 &b^2 &0 &-b^4\\
 -b^2 &0 &b^4 &0\\
 0 &-b^4 &0 &b^6
 \end{pmatrix}.$$
 Then we have
$$|\sigma(x-y\alpha)|=|\sigma(z)|\le |\sigma(z')||\sigma(\alpha)|\le c_{2}'^{\frac{1}{2}}|\sigma(\alpha)|,$$
where
$$c_{2}'=\mbox{max}\{h_2 (z_0,z_1,z_2,z_3): z_i=-\frac{1}{2}~\mbox{or} ~\frac{1}{2},i=0,1,2,3\}.$$

However
both
$$|z'|^2=(z_0,z_1,z_2,z_3)M
 H_1M^{T}
 \begin{pmatrix}
 z_0\\
 z_1\\
 z_2\\
 z_3
 \end{pmatrix}
 $$
 and
$$
|\sigma(z')|^2=(z_0,z_1,z_2,z_3)M
 H_2M^{T}
 \begin{pmatrix}
 z_0\\
 z_1\\
 z_2\\
 z_3
\end{pmatrix}
$$
are  positive definite quadratic forms determined by $a>0$ and $b>0$.
So  $|z'|$ and $|\sigma(z')|$ reach maximal value at same point.
Let $c_i= {c_{i}'}^{\frac{1}{2}}, i=1,2.$  Then
the proof is completed.
\end{proof}
\end{subsection}

\begin{subsection}{Construction of $W_m,C_m,G_m$}
Let $F=\mathbb{Q}\Big(\sqrt{-(D+B\sqrt{D}})\Big)$ be a cyclic quartic field with the class number $h(F)=1,$ and
let $S_{m+1}=\{v_1,v_2,\cdots,v_{m+1}\}$, where $v_{i}$ corresponds to the prime ideal $\mathcal{P}_i:=\mathcal{P}_{v_i}$ for
$i=1,2,\cdots,m+1.$
In order to use Theorem 1.1 to compute the tame kernel $K_2\mathcal{O}_F,$ we construct $W_m,C_m$ and $G_m$ as follows.

Firstly, by Lemma 3.1, for each $i$ there exists an  $\alpha_i\in{\mathcal{O}_F}$ satisfying
$\mathcal{P}_{i}=(\alpha_i)$ and $|\sigma(\xi)|\le \Big|\frac{\sigma(\alpha_i)}{\alpha_i}\Big|\le |\xi|,$ where $i=1,2,\cdots,m+1.$
Thus we define
$$W_{m}=\{\alpha_1,\alpha_2,\cdots,\alpha_m\}\bigcup \{-1,\xi\}.$$
Clearly, from the construction of $W_m,$ we know immediately  that $U_m$ can be generated by $W_m.$

Secondly, let
\begin{equation*}
\begin{split}
C'_{m}=\{c\in{\mathcal{O}_K}: \ |c|\le c_1|\alpha_{m+1}|, |\sigma(c)|\le c_2|\sigma(\alpha_{m+1})|\}.
\end{split}
\end{equation*}
Then the set $C_{m}$ defined to be a subset of  $C'_{m}$ such that $1\in{C_{m}},~0\notin{C_{m}}$ and $c_1-c_2\notin
\mathcal{P}_{m+1}$ for any $c_1,c_2\in{C_m}.$ Clearly, we have  $1\in C_m\cap Ker\beta\subseteq U'_1,$ which implies that  condition (iii) of Theorem 1.1 is satisfied. In the following, we will prove that there must exist a $C_m$ which satisfies condition I and condition II further.

Finally, let $\delta:=\big(\frac{2}{\pi}\big)^\frac{1}{2}|D|^\frac{1}{8}$ and define
\begin{equation*}
\begin{split}
G'_{m}=\Big\{g\in{\mathcal{O}_K}: |g|\le \delta N(\mathcal{P}_{m+1})^{\frac{1}{8}}, |\sigma(g)|\le \delta N(\sigma(\mathcal{P}_{m+1}))^{\frac{1}{8}}\Big\}.
\end{split}
\end{equation*}
When $N(\mathcal{P}_{m+1})>\delta^8,$ by GTT theorem and the proof of Lemma 1.2 in [10], there exists a subset $G_m\subseteq U_m$ with $G_m\subseteq G'_m$ such that $k^{*}_v$ can be generated by $\beta(G_m),$ which means the second part of condition (ii) in Theorem 1.1 is satisfied.
\end{subsection}

\begin{subsection}{Theoretical bounds}
\begin{subsubsection}{The bounds in imaginary cyclic quartic field case}

 The following lemma is very helpful.
 \begin{lemma}\quad
 	Suppose that the elements $a,b\in\mathcal{O}_F \bigcap U_{m}$ satisfy the conditions $a\equiv b(mod\, \mathcal{P}_{m+1})$ and $N(a-b)<N^2(\mathcal{P}_{m+1}).$ Then $\frac{a}{b}\in{U'_1}.$
 \end{lemma}
 \begin{proof}\quad
 	See Claim 2 in the proof of Lemma 3.4 in [12].
 \end{proof}
 
 Define
 $$c'=\mbox{max} \Big\{c_1\frac{|\sigma(\xi)|}{|\xi|}+c_2\frac{|\xi|}{|\sigma(\xi)|}, c_2\frac{|\sigma(\xi)|}{|\xi|}+c_1\frac{|\xi|}{|\sigma(\xi)|} \Big\},$$
 
 \begin{lemma}\quad
 	If $N(\mathcal{P}_{m+1})\ge \Big(1+c_1c_2+c'\Big)^2,$ then $W_m\subseteq{C_mU'_1},$ i.e., condition I is satisfied.
 \end{lemma}
 \begin{proof}\quad
 	By  Lemma 3.3, if for any $w\in W_m$ there always exists a $c\in C_m$ satisfying $c\equiv w(\mbox{mod}\, \mathcal{P}_{m+1})$ and $N(w-c)<N^2(\mathcal{P}_{m+1}),$ then we have $W_m\subseteq C_mU'_1.$ So it suffices to investigate when the inequality $N(w-c)<N^2(\mathcal{P}_{m+1})$ holds.
 	
 	However, we have
 	\begin{equation*}
 	\begin{split}
 	N(w-c)&=|w-c||\sigma(w)-\sigma(c)||\sigma^2(w)-\sigma^2(c)||\sigma^3(w)-\sigma^3(c)|\\
 	&=|w-c||\bar w-\bar c||\sigma(w)-\sigma(c)||\overline{\sigma(w)}-\overline{\sigma}|\\
 	&=(|w-c||\sigma(w)-\sigma(c)|)^2.
 	\end{split}
 	\end{equation*}
 	We will estimate the term $|w-c||\sigma(w)-\sigma(c)|.$
 	
 	First we have
 	\begin{equation*}
 	\begin{split}
 	&|w-c||\sigma(w)-\sigma(c)|
 	=|w\sigma(w)-w\sigma(c)-c\sigma(w)+c\sigma(c)|\\
 	\le&|w\sigma(w)|+|w\sigma(c)|+|c\sigma(w)|+|c\sigma(c)|\\
 	=&N^{\frac{1}{2}}(w)+|w\sigma(c)|+|c\sigma(w)|+N^{\frac{1}{2}}(c).
 	\end{split}
 	\end{equation*}
 	From the construction of $W_m$ and $C_m$, we have $N^{\frac{1}{2}}(w)\le N^{\frac{1}{2}}(\mathcal{P}_{m+1})$ and $N^{\frac{1}{2}}(c)\le c_1c_2N^{\frac{1}{2}}(\mathcal{P}_{m+1}).$
 	
 	Now we  estimate the term $|w\sigma(c)|+|c\sigma(w)|=|c\sigma(w)|+\frac{N^{\frac{1}{2}}(w)N^{\frac{1}{2}}(c)}{|c\sigma(w)|}.$  By Lemma 3.1, for any $c\in C_m$ we have
 	\begin{equation*}
 	\begin{split}
 	|c|\le c_1|\alpha_{m+1}|\le c_1\frac{N^{\frac{1}{4}}(\alpha_{m+1})}{|\sigma(\xi)|^{\frac{1}{2}}}.
 	\end{split}
 	\end{equation*}
 	In virtue of $\Big|\frac{\sigma(\alpha_{m+1})}{\alpha_{m+1}}\Big|\le |\xi|$ and $|N(\alpha_{m+1})|=|\alpha_{m+1}|^2|\sigma(\alpha_{m+1})|^2,$ we have
 	\begin{equation*}
 	\begin{split}
 	|c|=&\frac{N^{\frac{1}{2}}(c)}{|\sigma(c)|}\ge \frac{N^{\frac{1}{2}}(c)}{c_2|\sigma(\alpha_{m+1})|}
 	=\frac{N^{\frac{1}{2}}(c)|\alpha_{m+1}|}{c_2|N^{\frac{1}{2}}(\alpha_{m+1})|}\\
 	\ge& \frac{N^{\frac{1}{2}}(c)N^{\frac{1}{4}}(\alpha_{m+1})}{c_2|\xi|^{\frac{1}{2}}N^{\frac{1}{2}}(\alpha_{m+1})}
 	=\frac{N^{\frac{1}{2}}(c)}{c_2|\xi|^{\frac{1}{2}}N^{\frac{1}{4}}(\alpha_{m+1})}.
 	\end{split}
 	\end{equation*}
 	So we have
 	\begin{equation*}
 	\begin{split}
 	\frac{N^{\frac{1}{2}}(c)}{c_2|\xi|^{\frac{1}{2}}N^{\frac{1}{4}}(\alpha_{m+1})}\le |c|\le \frac{c_1N^{\frac{1}{4}}(\alpha_{m+1})}{|\sigma(\xi)|^{\frac{1}{2}}}.
 	\end{split}
 	\end{equation*}
 	When $w \ne \xi\in W_m,$ from the construction of $W_m,$ we have
 	\begin{equation*}
 	\begin{split}
 	\frac{|\sigma(\xi)|}{|\xi|^\frac{1}{2}}N^{\frac{1}{4}}(w)\le |\sigma(\xi)||w|\le |\sigma(w)| \le|\xi||w| \le \frac{|\xi|}{|\sigma(\xi)|^{\frac{1}{2}}}N^{\frac{1}{4}}(w).
 	\end{split}
 	\end{equation*}
 	When $w=-1$ or $\xi,$ clearly the inequality above also holds. Thus we get
 	\begin{equation*}
 	\begin{split}
 	\frac{|\sigma(\xi)|N^{\frac{1}{2}}(c)N^{\frac{1}{4}}(w)}{c_2|\xi|N^{\frac{1}{4}}(\alpha_{m+1})}\le |c\sigma(w)|\le c_1 \frac{|\xi|}{|\sigma(\xi)|}N^{\frac{1}{4}}(w)N^{\frac{1}{4}}(\alpha_{m+1}).
 	\end{split}
 	\end{equation*}
 	It is  easy to show the function $f(x)=x+\frac{N^{\frac{1}{2}}(w)N^{\frac{1}{2}}(c)}{x}$ meet its maximal value on the boundary.
 	However the function values of $f(x)$ on boundary $x=\frac{|\sigma(\xi)|N^{\frac{1}{2}}(c)N^{\frac{1}{4}}(w)}{c_2|\xi|N^{\frac{1}{4}}(\alpha_{m+1})}$ or $c_1 \frac{|\xi|}{|\sigma(\xi)|}N^{\frac{1}{4}}(w)N^{\frac{1}{4}}(\alpha_{m+1})$ can be computed as follows.
 	\begin{equation*}
 	\begin{split}
 	&\frac{|\sigma(\xi)|N^{\frac{1}{2}}(c)N^{\frac{1}{4}}(w)}{c_2|\xi|N^{\frac{1}{4}}(\alpha_{m+1})}+\frac{c_2|\xi|}{|\sigma(\xi)|}N^{\frac{1}{4}}(w)N^{\frac{1}{4}}(\mathcal{P}_{m+1})\\
 	&\le
 	c_1\frac{|\sigma(\xi)|}{|\xi|}\frac{N^{\frac{1}{2}}(\alpha_{m+1})N^{\frac{1}{4}}(\alpha_{m+1})}{N^{\frac{1}{4}}(\alpha_{m+1})}+\frac{c_2|\xi|}{|\sigma(\xi)|}N^{\frac{1}{2}}(\alpha_{m+1})\\
 	&=\Big(c_1\frac{|\sigma(\xi)|}{|\xi|}+c_2\frac{|\xi|}{|\sigma(\xi)|}\Big)N^{\frac{1}{2}}(\alpha_{m+1})
 	\end{split}
 	\end{equation*}
 	and
 	\begin{equation*}
 	\begin{split}
 	&\frac{c_1|\xi|}{|\sigma(\xi)|}N^{\frac{1}{4}}(w)N^{\frac{1}{4}}(\mathcal{P}_{m+1})+\frac{|\sigma(\xi)|N^{\frac{1}{2}}(c)N^{\frac{1}{4}}(w)}{c_1|\xi|N^{\frac{1}{4}}(\alpha_{m+1})}\\
 	&\le
 	\frac{c_1|\xi|}{|\sigma(\xi)|}N^{\frac{1}{2}}(\alpha_{m+1})+c_1c_2\frac{|\sigma(\xi)|}{c_1|\xi|}\frac{N^{\frac{1}{2}}(\alpha_{m+1})N^{\frac{1}{4}}(\alpha_{m+1})}{N^{\frac{1}{4}}(\alpha_{m+1})}\\
 	&=\Big(c_2\frac{|\sigma(\xi)|}{|\xi|} + c_1\frac{|\xi|}{|\sigma(\xi)|}\Big)N^{\frac{1}{2}}(\alpha_{m+1}).
 	\end{split}
 	\end{equation*}
 	So we have
 	\begin{equation*}
 	\begin{split}
 	|c\sigma(w)|+|w\sigma(c)|=|c\sigma(w)|+\frac{N^{\frac{1}{2}}(w)N^{\frac{1}{2}}(c)}{|c\sigma(w)|}
 	\le c'N^{\frac{1}{2}}(\alpha_{m+1}),
 	\end{split}
 	\end{equation*}
 	where $c'=\mbox{max} \Big\{c_1\frac{|\sigma(\xi)|}{|\xi|}+c_2\frac{|\xi|}{|\sigma(\xi)|}, c_2\frac{|\sigma(\xi)|}{|\xi|}+c_1\frac{|\xi|}{|\sigma(\xi)|} \Big\}.$
 	Summarily we get
 	\begin{equation*}
 	\begin{split}
 	|N(w-c)|&=(|w-c||\sigma(w)-\sigma(c)|)^2\\
 	&\le(N^{\frac{1}{2}}(w)+|w\sigma(c)|+|c\sigma(w)|+N^{\frac{1}{2}}(c))^2\\
 	&\le\Big(1+c_1c_2+c'\Big)^2N(\alpha_{m+1}).
 	\end{split}
 	\end{equation*}
 	So when $N(\alpha_{m+1})=N(\mathcal{P}_{m+1})>\Big(1+c_1c_2+c'\Big)^2$, we have $W_m\subseteq CU'_1$.
 \end{proof}
 \begin{lemma}\quad
 	If $N(\alpha_{m+1})>\Big(\frac{\delta\sqrt{c_1c_2}}{2}+\sqrt{\frac{c_1c_2\delta^2}{4}+\sqrt{c_1c_2}}\Big)^8,$ then $C_mG_m\subseteq C_mU'_1,$ i.e., condition II is satisfied.
 \end{lemma}
 \begin{proof}\quad
 	By Lemma 3.3, in order to prove $C_mG_m\subseteq C_mU'_1$, it is sufficient to prove that for any $c\in C_m$ and $g\in G_m$ there exists a $\tilde{c}\in C_m$ such that $cg\equiv \tilde{c}~(\mbox{mod}\,\mathcal{P}_{m+1})$ and $N(cg-\tilde{c})<N^2(\mathcal{P}_{m+1}).$ So we should investigate  when the inequality $N(cg-\tilde{c})<N^2(\mathcal{P}_{m+1})$ holds.
 	
 	Let $c,\tilde{c}\in C_m,g\in G_m,$ and let $M_1,M_2\in \mathbb{R}$ with the conditions:
 	\begin{equation*}
 	\begin{split}
 	N(c)\le M_1,~ ~N(\tilde{c})\le M_1, ~ ~ |g|\le M_2, ~ ~|\sigma(g)|\le M_2.
 	\end{split}
 	\end{equation*}	
 	Then
 	\begin{equation*}
 	\begin{split}
 	N^{\frac{1}{2}}(cg-\tilde{c})=&|cg-\tilde{c}||\sigma(c)\sigma(g)-\sigma(\tilde{c})|\\
 	\le&(|cg|+|\tilde{c}|)(|\sigma(c)\sigma(g)|+|\sigma(\tilde{c})|)\\
 	\le&|c\sigma(c)||g\sigma(g)|+M_2(|c\sigma(\tilde{c})|+|\tilde{c}\sigma(c)|)+|\tilde{c}\sigma(\tilde{c})|\\
 	=&N^{\frac{1}{2}}(c)N^{\frac{1}{2}}(g)+M_2(|c\sigma(\tilde{c})|+|\tilde{c}\sigma(c)|)+N^{\frac{1}{2}}(\tilde{c}).
 	\end{split}
 	\end{equation*}
 	
 	Let us estimate the term $|c\sigma(\tilde{c})|+|\tilde{c}\sigma(c)|=|c\sigma(\tilde{c})|+\frac{N^{\frac{1}{2}}(c)N^{\frac{1}{2}}(\tilde{c})}{|c\sigma(\tilde{c})|}.$
 	By the definition of $C_m$, it is obvious that $|c|\le c_1|\alpha_{m+1}|$ and $|c|=\frac{N^{\frac{1}{2}}(c)}{|\sigma(c)|}\ge \frac{1}{|c_2\sigma(\alpha_{m+1})|}N^{\frac{1}{2}}(c).$
 	So we have
 	\begin{equation*}
 	\begin{split}
 	\frac{N^{\frac{1}{2}}(c)}{c_2|\sigma(\alpha_{m+1})|}\le |c| \le c_1|\alpha_{m+1}|.
 	\end{split}
 	\end{equation*}
 	Similarly, we  have
 	\begin{equation*}
 	\begin{split}
 	\frac{N^{\frac{1}{2}}(\tilde{c})}{c_1|\alpha_{m+1}|}\le \frac{N^{\frac{1}{2}}(\tilde{c})}{|\tilde{c}|}\le |\sigma(\tilde{c})|\le c_2|\sigma(\alpha_{m+1})|.
 	\end{split}
 	\end{equation*}
 	Therefore
 	\begin{equation*}
 	\begin{split}
 	\frac{N^{\frac{1}{2}}(c)N^{\frac{1}{2}}(\tilde{c})}{c_{1}c_{2}N^{\frac{1}{2}}(\alpha_{m+1})}\le |c\sigma(\tilde{c})|\le c_1c_2N^{\frac{1}{2}}(\alpha_{m+1})
 	\end{split}
 	\end{equation*}

 	Let $f(x)=x+\frac{N^{\frac{1}{2}}(c)N^{\frac{1}{2}}(\tilde{c})}{x}.$ It is easy to show that $f(x)$ meets the maximal value at $x=c_{1}c_{2}N^{\frac{1}{2}}(\alpha_{m+1}).$ So
 	\begin{equation*}
 	\begin{split}
 	|c\sigma(\tilde{c})|+|\tilde{c}\sigma(c)|&=|c\sigma(\tilde{c})|+\frac{N^{\frac{1}{2}}(c)N^{\frac{1}{2}}(\tilde{c})}{|c\sigma(\tilde{c})|}\\
 	&\le c_{1}c_{2}N^{\frac{1}{2}}(\alpha_{m+1})+\frac{N^{\frac{1}{2}}(c)N^{\frac{1}{2}}(\tilde{c})}{ c_{1}c_{2}N^{\frac{1}{2}}(\alpha_{m+1})}\\
 	&\leq2c_1c_2N^{\frac{1}{2}}(\alpha_{m+1}).
 	\end{split}
 	\end{equation*}
 	Then
 	\begin{equation*}
 	\begin{split}
 	N^{\frac{1}{2}}(cg-\tilde{c})&\le N^{\frac{1}{2}}(c)N^{\frac{1}{2}}(g)+M_2(|c\sigma(\tilde{c})|+|\tilde{c}\sigma(c)|)+N^{\frac{1}{2}}(\tilde{c})\\
 	&\le M_1^\frac{1}{2}M_2^2+2c_1c_2M_2N^{\frac{1}{2}}(\mathcal{P}_{m+1})+M_1^{\frac{1}{2}}
 	\end{split}
 	\end{equation*}
 	
 	By the definition of $C_m$ and $G_m,$ we can take $M_1=c_1^2c_2^2N(\alpha_{m+1})$ and $M_2=\delta N^{\frac{1}{8}}(\alpha_{m+1}).$
 	Hence we have
 	\begin{equation*}
 	\begin{split}
 	N^{\frac{1}{2}}(cg-\tilde{c})&\le
 	c_1c_2\delta^2N^{\frac{3}{4}}(\alpha_{m+1})+2\delta c_1c_2N^{\frac{5}{8}}(\alpha_{m+1})+c_1c_2N^{\frac{1}{2}}(\alpha_{m+1}).
 	\end{split}
 	\end{equation*}
 	So it is sufficent to consider the inequatity
 	\begin{equation*}
 	\begin{split}
 	c_1c_2\delta^2N^{\frac{3}{4}}(\alpha_{m+1})+2\delta c_1c_2N^{\frac{5}{8}}(\alpha_{m+1})+c_1c_2N^{\frac{1}{2}}(\alpha_{m+1})<N(\alpha_{m+1}),
 	\end{split}
 	\end{equation*}
 	i.e. $$N^{\frac{1}{2}}(\alpha_{m+1})-c_1c_2\delta^2N^{\frac{1}{4}}(\alpha_{m+1})-2\delta c_1c_2N^{\frac{1}{8}}(\alpha_{m+1})-c_1c_2>0.$$
 	This implies that when $N(\alpha_{m+1})>\Big(\frac{\delta\sqrt{c_1c_2}}{2}+\sqrt{\frac{c_1c_2\delta^2}{4}+\sqrt{c_1c_2}}\Big)^8,$ we have $N^{\frac{1}{2}}(cg-\tilde{c})<N(\alpha_{m+1}),$ as required.
 \end{proof}

\end{subsubsection}

\begin{subsubsection}{Groenewegen's general bound}
For any $m\in{\mathbb{Z}^{+}},$ we denote
$$K_2U_m:=(U_{m} \otimes U_{m})/\langle a\otimes b |a, b\in{U_{m}},\  a+b=1\ \mbox{or}  \ a+b=0\rangle$$
and
$$K_{2}^{(m)}\mathcal{O}_F=\mbox{ker}\Big(K_2U_m\rightarrow \bigoplus_{Nv\leq{Nv_{m}}}k_{v}^{*}\Big).$$
It is clear that there is a natural map $K_2U_m\rightarrow K_{2}F.$
Moreover we write
$$c_{F}=\mbox{max}\{2^{2n}\rho d^{2},2^{2n/3},\rho^{1/3}(d\tilde{d}^2)^{2/3},\rho d^3\}$$
where
$$d=\frac{2^{n}\Gamma{(\frac{n+2}{2})}}{(\pi n)^{n/2}}|\Delta|^{1/2} , \ \  \tilde{d}=\Big(\frac{2}{\pi}\Big)^s |\Delta|^{1/2}$$
and $\rho$ is the packing density of an $n$-dimensional sphere.
In \cite{GROENEWEGEN001},  Groenewegen proved the following theorem.
\begin{theorem}\quad
	For every number field $F$,
	for $N v_{m}>c_{F},$ the image of $K_{2}^{(m)}\mathcal{O}_F$ in $K_{2}F$ is equal to the tame kernel of $F$.
\end{theorem}
\begin{remark} \quad For an imaginary cyclic quartic field $\mathbb{Q}\Big(\sqrt{-(D+B\sqrt{D})}\Big)$ of class number one (see Section 2),  by Theorem 3.6 we can get a common  bound of $m$ for both condition I and condition II.  But, from the computation of the next section, we know that for condition I the bound obtained by Lemma 3.4 is better than that  obtained by Theorem 3.6 except for the cases $B=6,D=37$ and $B=2,D=61$, and for condition II, the bound obtained by Theorem 3.6 is  better than that obtained by Lemma 3.5 except for the case
	$B=1,D=2.$
	The comparison of the results is listed in the following table.
\end{remark}
\begin{table}[!hbp]
	\caption{}\label{eqtable}
	\centering
	\ \ \ \ \ \ \ \ \ \ \ \begin{tabular}{|c|c|c|c|}
		\hline
		number field $F$ & Lemma 3.4 & Lemma 3.5 & Theorem 3.6\\
		\hline
		$B=1,D=2$ & $172.525$ & $3253.539$ & $16146.993$ \\
		\hline
		$B=2,D=13$ & $1173.677$ & $45879.279$ & $17321.1$ \\
		\hline
		$B=2,D=29$ & $48710.067$ & $1867701099.860$ & $192289.567$ \\
		\hline
		$B=6,D=37$& $5284749.383$ & $61546835.003$ & $399362.147$ \\
		\hline
		$B=2,D=53$ & $114166.647$ & $4086894943.478$ & $1173787.115$ \\
		\hline
		$B=2,D=61$ & $180648285.891$ & $1680328728.448$ & $1789580.481$ \\
		\hline
	\end{tabular}
\end{table}

\end{subsubsection}
\end{subsection}
\end{section}

\begin{section}{Decreasing the value $m$}

\begin{subsection}{The general idea}
Let  $F=\mathbb{Q}\Big(\sqrt{-(D+B\sqrt{D})}\Big)$ be an imaginary  cyclic quartic field with class number $h_F =1$ and $\xi$ the fundamental unit.

As Balabas and Gangl did, we also aim at decreasing theoretical bound of $m$ practically. The  general idea is as follows.

At first, by Lemma 3.2, we get
the constants $c_1,c_2,c'.$
Let $c''=$ min$\{(1+c_N+c')^2, c_F\}.$

If $c''\leq c_{F},$ there exists an $m_{1}\in{\mathbb Z^{+}}$ satisfying  $N(\mathcal{P}_{m_{1}})\leq c''$ and $N(\mathcal{P}_{m_{1}+1})> c''.$
Thus, by Lemma 3.4, for  $m\in{\mathbb Z^{+}}$
satisfying $m>m_{1}$ and $c''<N(\mathcal{P}_{m_{1}+1})\leq N(\mathcal{P}_{m}),$ condition I holds for $m_1+1$. We want to show that it holds also for $m_1.$

If $c_{F}\leq c'',$ there exists an $m_{1}'\in{\mathbb Z^{+}}$ satisfying  $N(\mathcal{P}_{m_{1}'})\leq c''$ and $N(\mathcal{P}_{m_{1}'+1})> c''.$ By Theorem 3.6, the image of $K_{2}^{(m_{1}'+1)}\mathcal{O}$ in $K_{2}F$ is equal to the tame kernel of $F$. However
it is obvious that the image of $K_{2}^{(m_{1}'+1)}\mathcal{O}$ in $K_{2}F$ is $\mbox{ker}\Big(\partial: K_2^{S_{m_{1}'+1}}(F)\longrightarrow \coprod_{v\in S_{m_{1}'+1}}k^*_v\Big).$
So it is necessary to show  condition I  holds for $m_{1}'.$
Without loss of generality, we denote $m_{1}'$ also by  $m_{1}.$

Similarly, from Lemma 3.5 or Theorem 3.6, there exists an $m_{2}\in{\mathbb Z^{+}}$ such that condition II holds for $m_2+1$. We want to show that condition II holds also for $m_2.$

Then, for $m=m_1$ (resp. $m_{2}),$
we will construct the subset $G_{m-1},W_{m-1}$ and $C_{m-1}$ satisfying  condition I (resp. condition II).

In this way, the value of $m$ can be decreased step by step.

\end{subsection}
\begin{subsection}{Checking $\partial_{m}$ to be an isomorphism}
\

Our idea for checking $\partial_{m}$ to be an isomorphism is described as follows.

\textbf{(I) Constructing the subset $W_{m-1}.$}

Let $\xi$ be the fundamental unit.
By (3.2), the subset
$$W_{m-1}=\{\alpha_1,\alpha_2,\cdots,\alpha_{m-1}\}\bigcup \{-1,\xi\}$$
needs to be defined, where $\alpha_i\in{\mathcal{O}_F}$ satisfies
$\mathcal{P}_{i}=(\alpha_i)$ and $|\sigma(\xi)|\le \Big|\frac{\sigma(\alpha_i)}{\alpha_i}\Big|\le |\xi|$ for each $i=1,2,\cdots,m-1.$

However, firstly
for some fixed $i\in\{1,2,\cdots,m-1\},$ we must confirm that the generator $\alpha_i$ of the prime ideal $\mathcal{P}_i$
satisfies that $\Big|\frac{\sigma(\alpha_i)}{\alpha_i}\Big|$ nearly equals $1.$
Fortunately, in the PARI library
the function \textbf{GEN bnfisprincipal0(GEN bnf, GEN x, long flag)} can return such a generator $\alpha_{i}$
for the prime ideal $\mathcal{P}_{i}$.
In fact, in the algorithm implemented by the above function, the generator has been reduced,
which means that $\Big|\frac{\sigma(\alpha_i)}{\alpha_i}\Big|$ nearly equals $1.$

Secondly,
we must get such $\alpha_i$ for each $i\leq m-1.$
Thus, we must get at first  the prime ideals whose norms are less than or  equal to the boundary determined by Lemma 3.4 (Lemma 3.5 respectively).
In fact, for each prime number $p\in{\mathbb{Z}}$, it is easy to determine its residue class degree $f_{p}$
and  to obtain the prime ideals above it
by the PARI function \textbf{GEN idealprimedec(GEN nf, GEN p, long f).}
So by iterating through the prime numbers which can be factored  into the prime ideals with norm less than the boundary,
we can get the required $\alpha_i\in W_{m-1}$ for each $i=1,2,\cdots,m-1.$

\textbf{(II) Constructing the subset $G_{m-1}.$}

For the only element $g_{m-1}\in G_{m-1}$, we can know that

(i) $g_{m-1}(mod\, \mathcal{P}_{m})$ is the only generator of the multiplicative cyclic group $k^{*}_{v_{m}}$ of the residue class field $k_{v_{m}}$ by the second part of condition (ii) in Theorem 1.1;

(ii) the value $\big|\frac{g_{m-1}}{\sigma{(g_{m-1})}}\big|$  should nearly  equal $1, $ by the proof of Lemma 3.5.

In the case of $f_{v_{m}}=1$, it is obvious that
$$\langle g'_{m-1}(mod\, \mathcal{P}_{m}) \rangle =k^{*}_{v_{m}}\cong{\mathbb{Z}/(\mathcal{P}_{m}\cup{\mathbb{Z}})}= \langle g'_{m-1}(mod\, (\mathcal{P}_{m}\cup{\mathbb{Z}})) \rangle,$$
where $g'_{m-1}\in{\mathbb{Z}}.$
Set $g_{m-1}=g'_{m-1}.$ Then we can get $G_{m-1}=\{g_{m-1}\}$ with $\big|\frac{g_{m-1}}{\sigma{(g_{m-1})}}\big|=1.$

In the case of $f_{v_{m}}\neq 1,$ by the PARI function
\textbf{GEN Idealstar(GEN nf, GEN ideal, long flag)}, the generator $g_{m-1}(mod\, \mathcal{P}_{m})$ of the cyclic group $k^{*}_{v_{m}}$ can be obtained.
So we can set $G_{m-1}=\{g_{m-1}\}.$ It is easy to show  that  the above condition (i) and (ii) are satisfied for the only element $g_{m-1}$ of the set $G_{m-1}$.

\textbf{(III) Constructing the subset $C_{m-1}.$}

By (3.2), the subset $C_{m-1}$ contains the lifting of all elements of the multiplicative group $k^{*}_{v_{m}}$ and $1\in{F}.$
Moreover, by the proofs of Lemma 3.4 and Lemma 3.5, each element $c_{m-1}$ of the set $C_{m-1}$ should satisfy that
the value $\big|\frac{c_{m-1}}{\sigma{(c_{m-1})}}\big|$  nearly  equals $1.$

We can get the generator $g_{m-1}(mod\, \mathcal{P}_{m})$ of the group $k^{*}_{v_{m}},$
so each element $c_{m-1,i}(mod\, \mathcal{P}_{m})$ of the group $k^{*}_{v_{m}}$ can be expressed as
$$c_{m-1,i}(mod\, \mathcal{P}_{m})=(g_{m-1}(mod\, \mathcal{P}_{m}))^{i}$$
where $i=1,2,\cdots,N(v_m)-1.$
But it is difficult to find a lifting $c_{m-1,i}$ of the element $c_{m-1,i}(mod\, \mathcal{P}_{m})$,
which satisfies that the value $\big|\frac{c_{m-1,i}}{\sigma{(c_{m-1,i})}}\big|$  nearly  equals $1.$
The method we use to get a suitable  lifting can be shown as follows.

Firstly, let $c'_{m-1,i}=g^{i}_{m-1}-\beta\xi^{k}$ for each $i=1,2,\cdots,N(v_m)-1,$
where $\beta\in{\mathcal{O}_F}$ and $k$ is nonnegative integer.

Secondly, when $\beta$ runs through the elements of $\mathcal{O}_F$
in increasing order by  norm and $k$ runs through all nonnegative integers in increasing order,
we can determine whether $c'_{m-1,i}\in{\mathcal{P}_{m}}$ is true.
Thus we can get the minimum $\beta$ and $k$ such that $c'_{m-1,i}\in{\mathcal{P}_{m}}$
for each $i=1,2,\cdots,N(v_m)-1,$ and therefore $\beta\xi^{k}$ is a lifting of $c_{m-1,i}(mod\, \mathcal{P}_{m}).$ Hence, we can let $c_{m-1,i}=\beta\xi^{k}.$

Lastly, we can obtain the set $C_{m-1}=\{c_{m-1,i}|i=1,2,\cdots,N(v_m)-1\}\bigcup{\{1\}}.$

\textbf{(IV) Checking condition I(II).}

After obtaining the subsets $W_{m-1},G_{m-1}$ and $C_{m-1}$, now we can check condition I(II).
Fortunately for us, the PARI function \textbf{GEN bnfissunit(GEN bnf, GEN sfu, GEN x)} can help us to check
whether $\gamma\in{U_{m}}$ is true for some $\gamma\in{\mathcal{O}_F}.$
Thus it is easy to write programme to check condition I(II) for the finite prime place $v_{m}.$

Using the above ideas, we can design the software architecture and algorithms and write a programm to
compute some tame kernels $K_2\mathcal{O}_F$ for the cyclic quartic fields $F=\mathbb{Q}\Big(\sqrt{-(D+B\sqrt{D})}\Big)$
with class number one.
\end{subsection}

\begin{subsection}{Designing the classes}
\

It is well known that  Tate' theorem is right for any number field.
Thus we can build a software architecture to be extensible and reusable for computing the tame kernel of a general number field,
with the cases of imaginary cyclic quartic fields with class number one as examples.
So in the following computation, firstly we will focus on all objects instead of the process.

All of objects are as follows:

(1) the cyclic quartic field $F=\mathbb{Q}\Big(\sqrt{-(D+B\sqrt{D})}\Big)$;

(2) the prime ideal $v_m$ of the algebraic integral ring $\mathcal{O}_{F};$

(3) the verification method which is used in this section;

(4) the group of $S_{m}$-units
$U_{m}=\{a\in{F}|v(a)=0,v\not\in{S_{m}}\};$

(5) three subsets $C_{m-1}$, $W_{m-1}$ and $G_{m-1}$ of  $U_{m-1}$ corresponding to  $v_m$;

(6) the constants $c_1$, $c_2$  corresponding to $F.$

Then, according to the objects and the relationships among them, we design the following three classes:

(1) {\it CquarField} (an abstraction description of  the  field $F=\mathbb{Q}\Big(\sqrt{-(D+B\sqrt{D})}\Big)$;

(2) {\it Cideal} (an abstraction description of the prime ideal $v_m$);

(3) {\it Ccheck} (an abstraction description of the verification method).

Moreover,  the constants $c_1$, $c_2$ are regarded as the attributes of  $CquarField$
and  the sets $C_{m-1}$,$W_{m-1},G_{m-1}$ as the the attributes of  $Cideal;$
an object of {\it CquarField} is regarded as an attribute of  $Cideal,$
which  is actually an abstraction description about "prime ideal is subject to the cyclic quartic field $F$ ";
an object of  {\it CquarField} is also regarded  as an attribute of the class {\it Ccheck},
which means that "the verification method is corresponding to a given cyclic quartic field".

Summarily, the relations in the above descriptions can be  indicated by  the  static class diagram given in Figure 2.

\begin{remark} \quad The  reason why we use the Object-Oriented Programming(OOP) is that the architecture can be expanded.
	For example, if we can find a way to compute the tame kernel $K_{2}\mathcal{O}_{F_1}$ for another number field $F_1$,
	the only things we must do are:
	
	(1) creating a class $CF_1$ corresponding to  $F_1$;
	
	(2) creating a class $CF$ as the parent class of $CF_1$ and $CquarField$;
	
	(3) making an object of $CF$ as an attribute of $Cideal$ and $Ccheck$.
\end{remark}
Thus, we have complete the creation of  the embryonic form of the architecture.
The  last work is to implement the classes.

\begin{subsection}{The methods of the three classes}
By the theory of the Object Oriented Programming, a class is partitioned into three parts: the name, the attributes and the methods.

For the above three classes, we have designed  their methods, which are listed as follows (The algorithms implemented by these methods  will be described in the next section):

(i) The methods of  $CquarField$:
\begin{lstlisting}{language=C++}
/*return the constant c_1*/
GEN getc_1();
/*return the constant c_2*/
GEN getc_2();
/*return the transition matrix between the basises*/
GEN transMatrix();
/*return the bound determined by Lemma 3.4*/
GEN getBoundOne();
/*return the bound determined by Lemma 3.5*/
GEN getBoundTwo();
/*return all prime ideals whose norms are less than
*the bound which is determined by Lemma 3.4 (resp.Lemma 3.5)
*and corresponds to the parameter num_condition 1 (resp.2)*/
GEN getPrimeTable(int num_condition);
\end{lstlisting}

(ii) The methods of $Cideal$:
\begin{lstlisting}{language=C++}
/*return the set W_{m-1} corresponding to the prime
*ideal represented by the class Cideal*/
GEN getSetInitW();

/*return the set G_{m-1} corresponding to the prime
*ideal represented by the class Cideal*/
GEN getSetInitG();

/*return all ideals whose norms are less than
*the norm of the prime ideal represented by
*the class Cideal*/
GEN fgetAllideal();

/*return the set C_{m-1} corresponding to the prime
*ideal represented by the  class Cideal;
*this is the parent thread function*/
GEN Para_getSetInitC();

/*This is the child thread function*/
static void* Part_getSetInitC(void *arg);

/*check  condition I corresponding to the prime ideal
*represented by the class Cideal*/
bool checkConditionOne();

/*check condition II corresponding to the prime ideal
*represented by the class Cideal*/
bool checkConditionTwo();

/*return the set U_{m} corresponding to the prime
*ideal represented by the class Cideal*/
GEN getUm();
\end{lstlisting}

(iii) The methods of $Ccheck$:
\begin{lstlisting}{language=C++}
/*the parent thread function to check condition I */
bool Para_checkConditionOne(int num_thread);

/*the parent thread function to check condition II */
bool Para_checkConditionTwo(int num_thread);

/*the parent thread function to check condition I*/
static void* Part_checkConditionOne(void *arg);

/*the child thread function to check condition II*/
static void* Part_checkConditionTwo(void *arg).
\end{lstlisting}

\end{subsection}
\subsection{Create the sequence diagram thst shows the expected workflow}

In order to show the process of computting the tame kernel of the number field $F=\mathbb{Q}\Big(\sqrt{A(D+B\sqrt{D})}\Big),$ we create a sequence diagram given in Figure 3.

\begin{remark}{Some remark on the sequence diagram:}\quad
	
	Firstly, we create an object  of the class Ccheck, named as checker, by calling the constructed function \textbf{Ccheck::check(int a,int b,int c,int d)} of the class Ccheck, where the formal parameters a,b,c and d indicate the four parameters $A,B,C$ and $D$ of the cyclic quartic field  $F=\mathbb{Q}\Big(\sqrt{A(D+B\sqrt{D})}\Big),$ respectively. In the process, we create an object qfCom of the class CquarField, which indicates  the cyclic quartic field  $F=\mathbb{Q}\Big(\sqrt{A(D+B\sqrt{D})}\Big).$ Moreover, some important invariants, such as the fundamental  unit, the discriminant of the number field $F$ and so on, of the cyclic quartic field  $F$ are obtained.
	
	Secondly, after lots of tests we find some easy facts on the subset $C_{s-1}$ of  $U_{s-1}$ corresponding to  the prime place $v_s$ of the number field $F$ as follows.
	
	(1) In the process of obtaining the subsets $C_{s-1},$ $G_{s-1}$ and $W_{s-1}$ of $U_{s-1}$, the most difficult one is to obtain  $C_{s-1};$
	
	(2) The value of the theoretical bound, determined by the lemma 3.4, 3.5 and theorem 3.6,
	is very large. So the number of the sets $C_{s-1}$ obtained by computing are also very large.
	
	(3) We suppose that some important information of tame kernel of the number field $F$ must be hidden in the subset $C_{s-1}$ of the set $U_{s-1}$ for every prime ideal $v_{s}$ of the number field $F.$\\
	Thus, it must  take a long time to obtain the set $C_{s-1}$ for every prime place $v_{s}$ of the number field $F$ whose  norm $N(v_{s})$ is less than the theoretical bound. And  we  think that it is a  good idea to obtain the sets $C_{s-1}$ prior to the sets $G_{s-1}$ and $W_{s-1}$. Moreover, for finding more information on the tame kernel of the number field $F$ from those sets, we also hope that all of the obtained sets $C_{s-1}$  are preserved in persistent storage. Then, built on the above ideals, for every prime place $v_s$ of the number field $F$ whose  norm $N(v_{s})$ is less than the theoretical bound, after finishing creating the object checker, the sets  $C_{s-1}$ are needed to get as follows.
	
	(step 1) In order to obtain all of the set $C_{s-1}$,  the method \textbf{bool Ccheck::Para\_g\\etSetC(int num\_thread, int num\_threadf)} is called, where the first(second) parameter means how many the threads are used for obtaining the set $C_{s-1}$ corresponding to the prime ideal $v_{s}$ with residue class degree $f_{v_{s}}=1$($f_{v_{s}}\neq1$).
	
	(step 2) But the number of  $v_{s}$, with norm $N(v_{s})$ less than the  theoretical bound, is very large. Then  a technology of the parallel computing is needed. The method \textbf{void* Ccheck:: Part\_getSetC(void *arg)} is child thread function.
	
	(step 3) In the process of calling the method \textbf{Para\_getSetC(int num\_thread,\ \ int num\_threadf)} to obtain all sets $C_{s-1}$, we must finish the following two works. One hand, it is necessary to get all prime numbers corresponding to the prime ideals  of the number field $F$ whose norms are less than  the theoretical bound, for which the method \textbf{GEN CquarField::getPrimeTable(int num\_condition)} is designed; On the other hand, by the definition of $C'_{s}$(3.2), we also must obtain all ideals of the number field $F$ whose norms  are less than $c_1c_2N(v_{s+1}).$ However, if using the PARI library function \textbf{GEN ideallist0(GEN nf, long bound, long flag)}, it will take a very long time to realizes the capability  because the value $c_1c_2N(v_{s+1})$ is too large. For example, in the case of $F=\mathbb{Q}\Big(\sqrt{-(13+2\sqrt{13})}\Big),$   the theoretical bound is $45879$ and we must take about $2.5$ hours to obtain the ideals mentioned above; and in the case of $F=\mathbb{Q}\Big(\sqrt{-(29+2\sqrt{29})}\Big),$  the theoretical bound is $192289$ and we must take about $60$ hours. Moreover, there is no PARI function that returns all ideals whose norms are  some $n\in{\mathbb{Z}}$ in pari library. So in order to minimize the consumption of time we must make use of the parallel computing in this procedure. Thus the methods \textbf{GEN CquarField::para\_getAllideal(long num\_thread,long num\_condition)} and \textbf{void* CquarField::Part\_getAll\\ideal(void *arg)} must be designed as the father thread function and the child thread function respectively. The parameter \textbf{num\_thread} means how many threads can be used for computing those ideals, and the parameter \textbf{num\_condition} means on which condition the ideals are computed. By the two methods, it  takes only about $10$ minutes(resp. $1.5$ hours) to obtain the ideals when $F=\mathbb{Q}\Big(\sqrt{-(13+2\sqrt{13})}\Big)$ (resp.$F=\mathbb{Q}\Big(\sqrt{-(29+2\sqrt{29})}\Big)$).
	
	(step 4) In this step, by going through all prime ideals with the norms less than  the theoretical bound, we obtain all of sets $C_{s-1}.$ However, in every loop, we must finish the following works. Firstly, calling the constructed function \textbf{Cideal::Cideal(\\CquarField* quarf, GEN gen\_prime, int i\_th)} we can create the object of the class Cideal corresponding to prime ideal of the number field $F$ whose norm is less than the theoretical bound; secondly, calling the method \textbf{GEN Cideal::getSet\\InitG()} we obtain a generator element of the cyclic group $k_{v_{s}}^{*}$; lastly, calling the father thread function \textbf{void Cideal::Para\_getSetInitC()} and the child thread function \textbf{void* Cideal::Part\_getSetInitC(void *arg)} we obtain the set $C_{s-1}$ and save as a text file.

	Finally,  to check  condition I(resp. II), we  use POSIX threads to design the father thread function \textbf{bool Ccheck::Para\_checkConditionOne(int num\_thread)} (resp. \textbf{bool  Ccheck::Para\_checkConditionTwo(int num\_th-read)}) and the child thread function \textbf{void* Ccheck::Part\_checkConditionOne\ \ (void *arg)}(resp. \textbf{void* Ccheck::Part\_checkConditionTwo(void *arg)})  to check condition I(resp. II).
\end{remark}
\end{subsection}
\begin{subsection}{The methods of the three classes}
By the theory of the Object Oriented Programming, a class is partitioned into three parts: the name, the attributes and the methods.

 For the above three classes, we have designed  their methods, which are listed as follows (The algorithms implemented by these methods  will be described in the next section):

(i) The methods of  $CquarField$:
\begin{lstlisting}{language=C++}
/*return the constant c_1*/
GEN getc_1();
/*return the constant c_2*/
GEN getc_2();
/*return the transition matrix between the basises*/
GEN transMatrix();
/*return the bound determined by Lemma 3.4*/
GEN getBoundOne();
/*return the bound determined by Lemma 3.5*/
GEN getBoundTwo();
/*return all prime ideals whose norms are less than
*the bound which is determined by Lemma 3.4 (resp.Lemma 3.5)
*and corresponds to the parameter num_condition 1 (resp.2)*/
GEN getPrimeTable(int num_condition);
\end{lstlisting}

(ii) The methods of $Cideal$:
\begin{lstlisting}{language=C++}
/*return the set W_{m-1} corresponding to the prime
*ideal represented by the class Cideal*/
GEN getSetInitW();

/*return the set G_{m-1} corresponding to the prime
*ideal represented by the class Cideal*/
GEN getSetInitG();

/*return all ideals whose norms are less than
*the norm of the prime ideal represented by
*the class Cideal*/
GEN fgetAllideal();

/*return the set C_{m-1} corresponding to the prime
*ideal represented by the  class Cideal;
*this is the parent thread function*/
GEN Para_getSetInitC();

/*This is the child thread function*/
static void* Part_getSetInitC(void *arg);

/*check  condition I corresponding to the prime ideal
*represented by the class Cideal*/
bool checkConditionOne();

/*check condition II corresponding to the prime ideal
*represented by the class Cideal*/
bool checkConditionTwo();

/*return the set U_{m} corresponding to the prime
*ideal represented by the class Cideal*/
GEN getUm();
\end{lstlisting}

(iii) The methods of $Ccheck$:
\begin{lstlisting}{language=C++}
/*the parent thread function to check condition I */
bool Para_checkConditionOne(int num_thread);

/*the parent thread function to check condition II */
bool Para_checkConditionTwo(int num_thread);

/*the parent thread function to check condition I*/
static void* Part_checkConditionOne(void *arg);

/*the child thread function to check condition II*/
static void* Part_checkConditionTwo(void *arg).
\end{lstlisting}
\end{subsection}

\begin{subsection}{The algorithms implemented by  the methods.}
\begin{subsubsection}{Some frequently-used algorithms}
	
During decreasing the value $m$, there are three things we must  compute from time to time. The first one is to decompose  a (positive) prime number $p$ into prime ideals in the cyclic quartic field $F$, the second one is to obtain  the generators of an ideal in the cyclic quartic field $F$ and the third one is to get the condition of determining whether an element of $\mathcal{O}_F$ is in the group $U_{m-1}.$ However,  the three things can be done by using Lemma 4.2, and Algorithm 4.1 and Algorithm 4.2 below. Moreover, as is well known, Lemma 4.1 can be  implemented by the PARI's functions \textbf{GEN idealprimedec(GEN nf, GEN p)} and Algorithm 4.1 and Algorithm 4.2
by \textbf{GEN bnfisprincipal0(GEN bnf, GEN x, long flag)}
and
\textbf{GEN bnfissunit(GEN bnf, GEN sfu, GEN x)}.

\begin{lemma}[Theorem 4.8.13 (\cite{cohen138})]\quad
	Let $F=\mathbb{Q}(\theta)$ ba a number field, where $\theta$ is an algebiaic integer, whose minimal polnomial is denoted $T(X)$. Let $f$ be the index of $\theta$. Then for any prime $p$ not dividing $f$ one can obtain the prime decomposition of $p\mathcal{O}_{F}$ as follows. Let
	$$T(X)\equiv \prod_{1}^{g}T_{i}(X)^{e_{i}}\pmod p$$
	be the decomposition of $T$ into irriducible factors in $\mathbb{F}_{p}[X],$ where the $T_{i}(X)$ are taken to be monic. Then
	$$p\mathcal{O}_{F}=\prod_{i=1}^{g}{\mathcal{P}_{i}}^{e_{i}},$$
	where
	$$\mathcal{P}_{i}=(p,T_{i}(\theta))=p\mathcal{O}_{F}+T_{i}(\theta)\mathcal{O}_{F}.$$
	Furthermore, the residual index $f_i$ is equal to the degree of $T_i(X).$
\end{lemma}

\begin{algorithm}[!htb]
	\caption{(Algorithm 6.5.10 (\cite{cohen138}))}
	\begin{algorithmic}
		\REQUIRE
		Given an ideal $I$ of $\mathcal{O}_F$ for a number field $F=\mathbb{Q}(\theta).$\\
		\ENSURE  Test whether $I$ is a principal ideal, and if it is,compute an $\alpha\in{F}$ such that $I=\alpha\mathcal{O}_F.$\\
		\STATE \textbf{1.[reduce to primitive]} If $I$ is not a primitive integral ideal, compute a rational number $a$ such that $I/(a)$ is primitive integral, and set $I\leftarrow I/(a).$\\
		\STATE \textbf{2.[Small norm]} If $N(I)$ is divisible only by prime numbers below the prime ideals in the factor base, set $v_{i}\leftarrow 0$ for $i<s,\beta\leftarrow a$ and go to step $4.$\\
		\STATE \textbf{3.[Generate random relations]} Choose random nonnegative integers $v_{i}<20$ for $i<s$,compute the ideal $I_{1}\leftarrow I\prod_{1\leq{i}\leq{s}}S_{i}^{v_i},$ and let $J=I_{1}/(\gamma)$ be the ideal obtained by LLL-reducing $I_1$ along the direction of the zero vector. If $N(J)$ is divisible only bt the prime numbers less than equal to $L_1$, set $I\leftarrow J, \beta\leftarrow a\gamma$ and go to step 4. Otherwise, go to step 3.\\
		\STATE \textbf{4.[Factor $I$]} Using Algorithm 4.8.17 in \cite{cohen138}, factor $I$ on the factor base FB. Let $I=\prod_{1\leq{i}\leq{k}}p_{i}^{x_i}.$ Let $X$(resp.Y) be the column vector of the $x_i-v_i$ for $i\leq r$(resp.$i>r$), where $r$ is the number of rows of the matrix $B,$ as above, and where we set $v_i=0$ for $i>s.$\\
		\STATE \textbf{5.[Check if principal]} Let $Z\leftarrow D^{-1}U(X-BY)$(since $D$ is a diagonal matrix, no matrix inverse must be computed here). If some entry of $Z$ is not integral, output a message saying that the ideal $I$ is not a principal ideal and terminate the algorithm.\\
		\STATE \textbf{6.[Use Archimedean information]} Let A be the ($c_1+k$)-column vector whose first $c_1$ elements are zero, whose next $r$ elements are the elements of Z, and whose last $k-r$ elements are element of $Y.$ Let $A_{C}=(a_{i})_{1\leq{i}\leq{r_{u}}}\leftarrow{M_{C}^{''}A}.$\\
		\STATE \textbf{7.[Restore correct information]} Set $s\leftarrow(\mbox{ln}N(I)/n),$ and let $A'=(a_{i}')_{1\leq{i}\leq{n}}$ be defined by $a_{i}'\leq{\mbox{exp}(s+a_i)}$ if $i\leq{r_1},a_{i}'\leftarrow \mbox{exp}(s+\bar{(a_{i-r_2})})$ if $r_{u}<i\leq n$\\
		\STATE \textbf{8.[Round]} Set $A''\leftarrow{\Omega^{-1}A'}$ where $\Omega=\sigma_j(\omega_i)$ as in Algorithm 6.5.8 in\cite{cohen138}. The coefficients of $A''$ must be close to rational integers. If this is not the case, then either the precision used to make the computation was insufficient or the desired $\alpha$ is too large. Otherwise, round the coefficients of $A''$ to the nearest integer.\\
		\STATE \textbf{9.[Terminate]} Let $\alpha '$ be the element of $\mathcal{O}_F$ whose coordinates in the integral basis are given by the vector $A''.$ Set $\alpha\leftarrow{\beta\alpha '}.$ If $I\neq\alpha\mathcal{O}_F,$ output an error message stating that the accuracy is not sufficient to compute $\alpha.$ Otherwise, output $\alpha$ and terminate the algorithm.
	\end{algorithmic}
\end{algorithm}

\begin{algorithm}[!htb]
	\caption{(Algorithm 7.4.8 (\cite{cohen193}))}
	\begin{algorithmic}
		\STATE Let Cl$(F)=(B,D_{B})$ be the SNF of the class group of $F$, where $B=(\bar{\mathfrak{b}_{i}})$ and the $\mathfrak{b}_i$ are the ideals of $F.$ The algorithm computes algebraic integers $\gamma_{i}$ for $1\leq{i}\leq{s}$
		such that $U_{S}(F)=U(F)\oplus_{1\leq{i}\leq{s}}{\mathbb{Z}\gamma_{i}}.$
		We let $\mathfrak{p}_{i}$ be the prime ideals of $S.$\\
		\STATE\textbf{1.[Compute discrete logarithms]} Using the principal ideal algorithm, compute the matrix $P$ whose columns are the discrete logarithms of $\bar{\mathfrak{p}}$ with respect to $B,$
		for each $\mathfrak{p}\in{S}.$\\
		\STATE\textbf{2.[Compute big HNF]} Using one of the algorithms for HNF computations, compute the unimodular matrix $U=\Big(\begin{matrix}U_1 & U_2\\U_3 & U_4\end{matrix}\Big)$ such that $(P|D_{B})U=(0|H)$ with H in HNF.\\
		\STATE\textbf{3.[Compute $\gamma\mathcal{O}_{F}$]} Compute the HNF W of the matrix $U_1,$ and set $[a_1,a_2,\cdots,a_s]\leftarrow{[\mathfrak{p}_{1},\cdots,\mathfrak{p}_{s}]W}.$\\
		\STATE\textbf{4.[Find generators]}(Here the $a_{j}$ are principal ideals)Using the principal ideal algorithm again, for each $j,$ find $\gamma_{j}$ such that $a_j=\gamma_{j}\mathcal{O}_{F}.$ Output the $\gamma_{j}$ and terminate the algorithm.
	\end{algorithmic}
\end{algorithm}

\end{subsubsection}

\begin{subsubsection}{The algorithms implemented by the methods of  CquarField}
	
In the class\\ {\it Cquarfield}, by Lemma 2.3, it is easy to design an algorithm implemented by the method \textbf{GEN transMatrix()} which returns the transition matrix from the basis $1,\beta,\beta^2,\beta^3$ to the basis $\gamma_0,\gamma_1,\gamma_2,\gamma_3$;
similarly, by Lemma 3.2, it is also easy to design  algorithms implemented by  \textbf{GEN getc\_1()}, \textbf{GEN getc\_2()}  which return the constants $c_1$, $c_2$  corresponding to the field $F=\mathbb{Q}\Big(\sqrt{-(D+B\sqrt{D})}\Big);$ by Lemma 3.4 and Lemma 3.5, it is very easy to design an algorithm implemented by  the methods  \textbf{GEN getBoundOne()} and \textbf{GEN getBoundTwo()} which can be used to compute  the bounds for  condition I and condition II.

We must obtain all prime ideals whose  norms are less than the bounds for condition I and condition II, which can be realized by the method \textbf{GEN getPrimeTable(int num\_condition)} of {\it CquarField}.

In fact, let $b_1=N(v_{m_{0}}),$ $b_2=N(v_{m_{0}'}),$ and let
$$T_{F,i}=\{p=\mathfrak{P}\cap\mathbb{Z}\in{\mathbb{Z}}|N(\mathfrak{P})< b_i,\mathfrak{P}\in{\mbox{spec}\mathcal{O}_{F}}\},$$
$$T_{F,i}'=\{\mathfrak{P}\in{\mbox{spec}\mathcal{O}_{F}}|N(\mathfrak{P})<b_i\},\ \ i=1,2,\ \ \ \ \ \ \ $$
where  spec$\mathcal{O}_{F}$ denotes the set of  prime ideals of the cyclic quartic field $F.$
To obtain the above sets, we design the following Algorithm 4.3 which can be implemented by the method \textbf{GEN getPrimeTable(int num\_condition)} of {\it CquarField}.
\begin{algorithm}[!htb]
	\caption{(The algorithm on {\bf getPrimeTable()})}
	\begin{algorithmic}
		\STATE \textbf{1.[Obtain the bound on norm]} By Lemma 4.5, Lemma 4.6 and Theorem 4.7, the bound on norm $b_1$(resp.$b_2)$ satisfying  condition I (resp. condition II) can be obtained.\\
		\STATE \textbf{2.[Obtain the set $T_{F,1}$(resp.$T_{F,2}$)]} Using the PARI' function GEN factoru(ulong n), the factorization of n can be returned. Moreover, the result is a 2-component vector [P, E], where P and E  are the prime divisors of n and the valuation of $n$ at prime point $p$ respectively.\\
		\STATE \textbf{3.[Obtain the sets $T_{F,1}^{'}$(resp.$T_{F,2}^{'}$)]} Using Algorithm 1, for any $k=p^{s}\in{T_{F,1}^{'}}$(resp.$T_{F,2}^{'}),$ the factoration ideals $\mathcal{P}_{i}$ of $p$ can be obtained. By comparing the value $s$ with the norm of the ideal $\mathcal{P}_{i}$, the elements of the set $T_{F,1}'$(resp.$T_{F,2}'$) can be obtained.
	\end{algorithmic}
\end{algorithm}
 
\

\

\

\

\

\

\

\

\

\
\end{subsubsection}
\begin{subsubsection}{The algorithms implemented by the methods of  Cideal}
In the class {\it Cideal}, the method  \textbf{GEN getSetInitG()} can be used to compute the set $G_{m-1}$ corresponding to the object, prime ideal $v_m$, of \textbf{Cideal}. We can easily finish the codes of the method \textbf{GEN getSetInitG()} by using the  PARI's functions \textbf{GEN znstar(GEN n)} and
\textbf{GEN idealstar0(GEN nf, GEN I, long flag)}, because the two functions
 have implemented  the following well-known Algorithm 4.4.
\begin{algorithm}[!htb]
\caption{(Algorithm 4.2.2 (\cite{cohen193}))}
\begin{algorithmic}
\STATE Let $m_0=\prod_{\mathfrak{P}}\mathfrak{P}^{v_{\mathfrak{P}}}$ be an integral ideal, and assume that we are given the SNF of $(\mathcal{O}_{F}/\mathfrak{P}^{v_{\mathfrak{P}}})^{*}=(G_{\mathfrak{P}},D_{\mathfrak{P}}).$ The algorithm computes the SNF of $(\mathcal{O}_{F}/m_{0})^{*}.$\\
\STATE\textbf{1.[Compute $\alpha_{\mathfrak{P}}$ and $\beta_{\mathfrak{P}}$]} Using Extended Euclid Algorithm in Dedekind Domains (Algorithm 1.3.2 (\cite{cohen193})), compute $\alpha_{\mathfrak{P}}$ and $\beta_{\mathfrak{P}}$ such that $\alpha_{\mathfrak{P}}\in{m_{0}/\mathfrak{P}^{v_{\mathfrak{P}}}},\beta_{\mathfrak{P}}\in{\mathfrak{P}^{v_{\mathfrak{P}}}}$ and $\alpha_{\mathfrak{P}}+\beta_{\mathfrak{P}}=1.$\\
\STATE\textbf{2.[Terminate]} Let $G$ be the concatenation of the $\beta_{\mathfrak{P}}1_{\mathcal{O}_{F}}+\alpha_{\mathfrak{P}}G_{\mathfrak{P}}$ and let $D$ be the diagonal concatenation of the SNF matrices $D_\mathfrak{P}.$ Using the algorithm of SNF for Finite Groups (Algorithm 4.1.3 (\cite{cohen193})) on the system of generators and relations $(G,D),$ output the SNF of the group $(\mathcal{O}_{F}/m_{0})^{*}$ and the auxiliary matrix $U_{\alpha},$ and terminate the algorithm.
\end{algorithmic}
\end{algorithm}

In order to obtain the set $W_{m-1}$ corresponding to the prime ideal $v_{m}$, we design the method \textbf{GEN getSetInitW()} of  {\it Cideal}. In the process of realizing this method, we use the PARI's function \textbf{GEN ideallist0(GEN nf, long bound, long flag)}, because they have implemented the following Algorithm 4.5 and returned all ideals whose norms are less than the value {\bf bound}.
We also give Algorithm 4.6, which outputs the set $W_{m-1}$ and is implemented by the method \textbf{GEN getSetInitW()}, as follows.

\

\

\

\

\

\

\

\begin{algorithm}[!htb]
\caption{(Algorithm 2.3.23 (\cite{cohen193}))}
\begin{algorithmic}
\STATE Let $K$ be a number field and $B$ be a positive integer. The algorithm outputs a list $\mathcal{L}$ such that for each $n\leq{B},$ $\mathcal{L}_{n}$ is the list of all integral ideals of absolute norm equal to $n.$\\
\STATE \textbf{1.[Initialize]}For $2\leq{n}\leq{B}$ set $\mathcal{L}_{N}\leftarrow{\emptyset},$ then set $\mathcal{L}_{1}\leftarrow{\mathcal{O}_{K}}$ and $p\leftarrow{0}.$\\
\STATE \textbf{2.[Next prime]}Replace $p$ by the smallest prime strictly larger than $p.$ If $p>B,$ output $\mathcal{L}$ and terminate the algorithm.\\
\textbf{3.[Factor $\mathbb{O}_{F}$]} Using Algorithm 6.2.9 in \cite{cohen138}, factor $p\mathcal{O}_K$ as $p\mathcal{O}_K=\prod_{1\leq{i}\leq{g}}\mathcal{P}_{i}^{e_{i}}$ with $e_{i}\geq{1},$ and let $f_{i}=f(\mathcal{P}_{i}|p).$ Set $j\leftarrow{0}.$\\
\STATE \textbf{4.[Next prime ideal]} Set $j\leftarrow{j+1}.$ If $j>g,$ go to step 2. Otherwise, set $q\leftarrow{p^{f_{j}}},n\leftarrow{0}.$\\
\STATE \textbf{5.[Loop through all multiples of $q$]} Set $n\leftarrow{n+q}.$ If $n>B,$ go to step 4. Otherwise,set $\mathcal{L}_{n}\leftarrow{\mathcal{L}_{n}\cup p_j\mathcal{L}_{n/q}},$ where $\mathcal{L}_{n}$ is the list of products by the ideal $p_j$ of the elements of
$\mathcal{L}_{n/q}$ and go to step 5.\\
\end{algorithmic}
\end{algorithm}

\begin{algorithm}
\caption{(The algorithm on {\bf getSetInitW()})}
\begin{algorithmic}
\STATE\textbf{1.[Initialize]} For the prime ideal $v_m$ set $W_{m-1}\leftarrow{\emptyset}$ and $\mathcal{I}\leftarrow{\emptyset}.$\\
\STATE\textbf{2.[Obtain ideals which nrom are less than $N(v_{m})$]} Using the PARI's function {\bf GEN ideallist0(GEN nf, long bound, long flag)}, all of ideals whose norms are less than $N(v_{m})$ can be obtained. Then put them into the set $\mathcal{I}$\\
\STATE\textbf{3.[Obtain all prime ideals whose norm are less than  $N(v_{m})$]} By looping through the set $\mathcal{I}$ and checking the structure of the ideal returned by PARI'function, we can get all prime ideals whose norm are less than  $N(v_{m})$.\\
\STATE\textbf{4.[Obtain the set $W_{m-1}$]} For the prime ideal $\mathcal{P}_{i},$ by using the PARI's function prime ideals whose norms are less than  $N(v_{m})$, the generator $\alpha_{i}$ can be returned. Then set $W_{m-1}\leftarrow{\alpha_{i}},$ where $i=1,2,\cdots,m.$
\end{algorithmic}
\end{algorithm}

In the process of obtaining  the sets $G_{m-1}$, $W_{m-1}$ and $C_{m-1}$, the most difficult thing is the computation of  $C_{m-1}$, because we  meet the following two difficulties:

(i)   The set $C_{m-1}$ is too large when $N(v_{m})$ is large since we have $|C_{m-1}|=N(v_{m})-1$;

(ii) we know that the set $C_{m-1}$ consists of the  representatives of some elements in $k_{v_{m}}^{*}$, but we can not ensure that the set $C_{m-1}$ must satisfy condition I and condition II under  arbitrary-chosen   representatives.

 To  overcome these difficulties, we  use the method of traversal but with the choice of representatives in a conjecturally right way.

Firstly, we find that the element $c\in{C_{m-1}}$ should be ``some shortest distance point" in ``some distance" of the set $c+v_{m}.$ So in order to look for the right $c\in{C_{m-1}}$, we set the range from an element whose norm is one.

Secondly,  it is  our choice  to take full advantage of multi-core processor hardware performance  to reduce the computation time. Thus,  we must use the technology of  the multi-threaded parallel computing to improve the speed of  Algorithm 4.7 obtaining the set $C_{m-1}$ as follows.
\

\

\
\begin{algorithm}
\caption{(The algorithm on {\bf getSetInitC()})}
\begin{algorithmic}
\STATE Let $v_{m}$ be a prime ideal of $\mathcal{O}_F.$ The algorithm outputs a set $C_{m-1}$ satisfying the following conditions:\\
(i) The set $C_{m-1}$ consists of the  representatives of some elements in $(k_{v_{m}})^{*}$;\\
(ii) For any element $c\in{C_{m-1}},$ the equation $N(c)=min\{N(c+t\alpha_{m})|t\in{\mathcal{O}_F}\}$  holds.\\
\STATE\textbf{1.[Initialize]} Set num\_good$\leftarrow0$ and $C_{m-1}\leftarrow{\emptyset}$. Invoking the methods \textbf{GEN getc\_1()}
and \textbf{GEN getc\_2()}, we can get the constant numbers $c_1$ and  $c_2,$ respectively.
Moreover, invoking the PARI's function \textbf{GEN ideallist0(GEN nf, long bound, long flag)} and \textbf{long pr\_get\_f(GEN pr)},
we can get the residue class degree $f_m$ of $v_m$ and the set $C'_{m-1}$ of all ideals whose norms are less than or equal to $(c_1c_2)^2N(v_m),$
 respectively. Invoking the method \textbf{GEN Cideal::getSetInitG()},
we can get the unique  element $g\in{G_{m-1}}.$
Lastly, set num\_C$'\leftarrow|C'_{m-1}|$ and num\_C$\leftarrow{N(V_m)-1}.$ \\
\STATE\textbf{2.[Compare num\_good with num\_C]} If num\_good=num\_C  holds, the algorithm return $C_{m-1}$ and is terminated.\\
\STATE\textbf{3.[Set the germs of $C_{m-1}$]} For $1\leq{i}\leq{num\_C}$, if $f_m=1$ set c\_i$\leftarrow{i}$.
Otherwise, set c\_i$\leftarrow{g^i}.$\\
\STATE\textbf{4.[Look for the appropriate elements in $C_{m-1}$]}
For $1\leq{j}\leq{num\_C'}$ and $1\leq{k}\leq{70},$  set $c'\_j\leftarrow C'_{m-1}[j].$
Then invoke the \textbf{PARI}'s function \textbf{long idealval(GEN nf, GEN x, GEN pr)}
to decide whether or not $c_i-c'_j\xi^k$ is in $U_m.$
Let the function's returned value be $b$.
If $b>0$, set $c_i\leftarrow{c'_j\xi^k}$ and $num\_good\leftarrow{num\_good+1},$ and go to step 2;
otherwise, set $j\leftarrow j+1$ and $k\leftarrow k+1.$
\end{algorithmic}
\end{algorithm}

 In the class {\it Cideal}, the methods \textbf{GEN Para\_getSetInitC()} and \textbf{static void* Part\_getSetInitC(void *arg)}
are the parent thread and the child thread respectively.
Using these methods, we can obtain the set $C_{m-1}$ corresponding to the prime ideal $v_{m}.$
Moreover we can change the number of the child threads with  different computer's hardware.

In the class {\it Cideal}, the two methods introduced above are
\textbf{bool Cideal::checkC\\
onditionOne()} and
\textbf{bool Cideal::checkConditionTwo()}.
As a result, conditions I and condition II can be verified respectively for the prime ideal $v_{m}$ by  the two methods,
which implement respectively Algorithm 4.8 and Algorithm 4.9 below.
\

\

\

\

\

\

\

\

\

\

\

\

\

\

\

\

\

\

\

\

\begin{algorithm}
\caption{(The algorithm on {\bf checkConditionOne()})}
\begin{algorithmic}
\STATE Let $v_{m}$ be a prime ideal of $\mathcal{O}_F.$ The algorithm check whether or not condition I is hold for $v_m.$\\
\STATE\textbf{1.[Initialize]} Set num\_good$\leftarrow0$. Invoking the methods \textbf{GEN getSetInitW()},
\textbf{GEN getSetInitC()}and \textbf{getUm()}, we can get the sets $W_{m-1},C_{m-1}$ and $U_m,$ respectively.
Moreover, we can get the cardinal numbers of the sets $W_{m-1},C_{m-1},$
denoted  by $num\_W$ and $num\_C,$ respectively.\\
\STATE\textbf{2.[Compare num\_good with num\_W]} If num\_good=num\_W  holds, the algorithm returns true and is terminated.\\
\STATE\textbf{3.[Loop through all element in SetW]} For $1\leq{i}\leq{num\_W}$, set w\_i$\leftarrow{W_{m-1}[i]}$.\\
\STATE\textbf{4.[Look for the appropriate element $c$ in setC]}
For $1\leq{j}\leq{num\_C},$ set c\_j$\leftarrow{C_{m-1}[j]}$.
Invoking the \textbf{PARI}'s function \textbf{GEN bnfissunit(GEN bnf, GEN sfu, GEN x)} and getting it's returned value $b$,
we can decide whether or not $\frac{w_i}{c_j}-1$ is in $U_m.$ More precisely,   if $b>0$, set $num\_good\leftarrow{num\_good+1}$ and go to step 2;
otherwise, set j$\leftarrow j+1.$
\end{algorithmic}
\end{algorithm}
\begin{algorithm}
\caption{(The algorithm on {\bf checkConditionTwo()})}
\begin{algorithmic}
\STATE Let $v_{m}$ be a prime ideal of $\mathcal{O}_F.$ The algorithm check whether or not condition II  holds for $v_m.$\\
\STATE\textbf{1.[Initialize]} Set num\_good$\leftarrow0$. Invoking the methods \textbf{GEN getSetInitG()},
\STATE\textbf{GEN getSetInitC()}and \textbf{getUm()}, we can get the sets $G_{m-1}=\{g\},C_{m-1}$ and $U_m,$ respectively.
Moreover, we can get the cardinal number of the set $C_{m-1}.$
  denoted  by $num\_C.$\\
\STATE\textbf{2.[Compare num\_good with num\_C]} If num\_good=num\_C  holds, the algorithm returns true and is terminated.\\
\STATE\textbf{3.[Loop through all element in SetC]} For $1\leq{i}\leq{num\_C}$, set c\_i$\leftarrow{C_{m-1}[i]}$.\\
\STATE\textbf{4.[Look for the appropriate element $c'$ in setC]}
For $1\leq{j}\leq{num\_C},$ set c'\_j$\leftarrow{C_{m-1}[j]}$.
Invoking the \textbf{PARI}'s function \textbf{GEN bnfissunit(GEN bnf, GEN sfu, GEN x)} and getting it's returned value $b$,
we can decide whether or not $\frac{c_i}{gc'_j}-1$ is in $U_m$  where $g$ is the unique element in set $G_{m-1}.$
More precisely,  if $b>0$, set $num\_good\leftarrow{num\_good+1}$ and go to step 2;
otherwise, set j$\leftarrow j+1.$
\end{algorithmic}
\end{algorithm}

\end{subsubsection}
\begin{subsubsection}{The algorithms implemented the methods of Ccheck}
In the class {\it Cideal}, the method  \textbf{GEN getSetInitG()} can be used to compute the set $G_{m-1}$ corresponding to the object, prime ideal $v_m$, of \textbf{Cideal}. We can easily finish the codes of the method \textbf{GEN getSetInitG()} by using the  PARI's functions \textbf{GEN znstar(GEN n)} and
\textbf{GEN idealstar0(GEN nf, GEN I, long flag)}, because the two functions
have implemented  the following well-known Algorithm 4.10.

\begin{algorithm}[!htb]
	\caption{(Algorithm 4.2.2 (\cite{cohen193}))}
	\begin{algorithmic}
		\STATE Let $m_0=\prod_{\mathfrak{P}}\mathfrak{P}^{v_{\mathfrak{P}}}$ be an integral ideal, and assume that we are given the SNF of $(\mathcal{O}_{F}/\mathfrak{P}^{v_{\mathfrak{P}}})^{*}=(G_{\mathfrak{P}},D_{\mathfrak{P}}).$ The algorithm computes the SNF of $(\mathcal{O}_{F}/m_{0})^{*}.$\\
		\STATE\textbf{1.[Compute $\alpha_{\mathfrak{P}}$ and $\beta_{\mathfrak{P}}$]} Using Extended Euclid Algorithm in Dedekind Domains (Algorithm 1.3.2 (\cite{cohen193})), compute $\alpha_{\mathfrak{P}}$ and $\beta_{\mathfrak{P}}$ such that $\alpha_{\mathfrak{P}}\in{m_{0}/\mathfrak{P}^{v_{\mathfrak{P}}}},\beta_{\mathfrak{P}}\in{\mathfrak{P}^{v_{\mathfrak{P}}}}$ and $\alpha_{\mathfrak{P}}+\beta_{\mathfrak{P}}=1.$\\
		\STATE\textbf{2.[Terminate]} Let $G$ be the concatenation of the $\beta_{\mathfrak{P}}1_{\mathcal{O}_{F}}+\alpha_{\mathfrak{P}}G_{\mathfrak{P}}$ and let $D$ be the diagonal concatenation of the SNF matrices $D_\mathfrak{P}.$ Using the algorithm of SNF for Finite Groups (Algorithm 4.1.3 (\cite{cohen193})) on the system of generators and relations $(G,D),$ output the SNF of the group $(\mathcal{O}_{F}/m_{0})^{*}$ and the auxiliary matrix $U_{\alpha},$ and terminate the algorithm.
	\end{algorithmic}
\end{algorithm}

In order to obtain the set $W_{m-1}$ corresponding to the prime ideal $v_{m}$, we design the method \textbf{GEN getSetInitW()} of  {\it Cideal}. In the process of realizing this method, we use the PARI's function \textbf{GEN ideallist0(GEN nf, long bound, long flag)}, because they have implemented the following Algorithm 4.5 and returned all ideals whose norms are less than the value {\bf bound}.
We also give Algorithm 4.6, which outputs the set $W_{m-1}$ and is implemented by the method \textbf{GEN getSetInitW()}, as follows.

\begin{algorithm}[!htb]
	\caption{(Algorithm 2.3.23 (\cite{cohen193}))}
	\begin{algorithmic}
		\STATE Let $K$ be a number field and $B$ be a positive integer. The algorithm outputs a list $\mathcal{L}$ such that for each $n\leq{B},$ $\mathcal{L}_{n}$ is the list of all integral ideals of absolute norm equal to $n.$\\
		\STATE \textbf{1.[Initialize]}For $2\leq{n}\leq{B}$ set $\mathcal{L}_{N}\leftarrow{\emptyset},$ then set $\mathcal{L}_{1}\leftarrow{\mathcal{O}_{K}}$ and $p\leftarrow{0}.$\\
		\STATE \textbf{2.[Next prime]}Replace $p$ by the smallest prime strictly larger than $p.$ If $p>B,$ output $\mathcal{L}$ and terminate the algorithm.\\
		\textbf{3.[Factor $\mathbb{O}_{F}$]} Using Algorithm 6.2.9 in \cite{cohen138}, factor $p\mathcal{O}_K$ as $p\mathcal{O}_K=\prod_{1\leq{i}\leq{g}}\mathcal{P}_{i}^{e_{i}}$ with $e_{i}\geq{1},$ and let $f_{i}=f(\mathcal{P}_{i}|p).$ Set $j\leftarrow{0}.$\\
		\STATE \textbf{4.[Next prime ideal]} Set $j\leftarrow{j+1}.$ If $j>g,$ go to step 2. Otherwise, set $q\leftarrow{p^{f_{j}}},n\leftarrow{0}.$\\
		\STATE \textbf{5.[Loop through all multiples of $q$]} Set $n\leftarrow{n+q}.$ If $n>B,$ go to step 4. Otherwise,set $\mathcal{L}_{n}\leftarrow{\mathcal{L}_{n}\cup p_j\mathcal{L}_{n/q}},$ where $\mathcal{L}_{n}$ is the list of products by the ideal $p_j$ of the elements of
		$\mathcal{L}_{n/q}$ and go to step 5.\\
	\end{algorithmic}
\end{algorithm}

\begin{algorithm}
	\caption{(The algorithm on {\bf getSetInitW()})}
	\begin{algorithmic}
		\STATE\textbf{1.[Initialize]} For the prime ideal $v_m$ set $W_{m-1}\leftarrow{\emptyset}$ and $\mathcal{I}\leftarrow{\emptyset}.$\\
		\STATE\textbf{2.[Obtain ideals which nrom are less than $N(v_{m})$]} Using the PARI's function {\bf GEN ideallist0(GEN nf, long bound, long flag)}, all of ideals whose norms are less than $N(v_{m})$ can be obtained. Then put them into the set $\mathcal{I}$\\
		\STATE\textbf{3.[Obtain all prime ideals whose norm are less than  $N(v_{m})$]} By looping through the set $\mathcal{I}$ and checking the structure of the ideal returned by PARI'function, we can get all prime ideals whose norm are less than  $N(v_{m})$.\\
		\STATE\textbf{4.[Obtain the set $W_{m-1}$]} For the prime ideal $\mathcal{P}_{i},$ by using the PARI's function prime ideals whose norms are less than  $N(v_{m})$, the generator $\alpha_{i}$ can be returned. Then set $W_{m-1}\leftarrow{\alpha_{i}},$ where $i=1,2,\cdots,m.$
	\end{algorithmic}
\end{algorithm}
\

\

\

\

\

\

\

In the process of obtaining  the sets $G_{m-1}$, $W_{m-1}$ and $C_{m-1}$, the most difficult thing is the computation of  $C_{m-1}$, because we  meet the following two difficulties:

(i)   The set $C_{m-1}$ is too large when $N(v_{m})$ is large since we have $|C_{m-1}|=N(v_{m})-1$;

(ii) we know that the set $C_{m-1}$ consists of the  representatives of some elements in $k_{v_{m}}^{*}$, but we can not ensure that the set $C_{m-1}$ must satisfy condition I and condition II under  arbitrary-chosen   representatives.

To  overcome these difficulties, we  use the method of traversal but with the choice of representatives in a conjecturally right way.

Firstly, we find that the element $c\in{C_{m-1}}$ should be ``some shortest distance point" in ``some distance" of the set $c+v_{m}.$ So in order to look for the right $c\in{C_{m-1}}$, we set the range from an element whose norm is one.

Secondly,  it is  our choice  to take full advantage of multi-core processor hardware performance  to reduce the computation time. Thus,  we must use the technology of  the multi-threaded parallel computing to improve the speed of  Algorithm 4.13 obtaining the set $C_{m-1}$ as follows.
\begin{algorithm}
	\caption{(The algorithm on {\bf getSetInitC()})}
	\begin{algorithmic}
		\STATE Let $v_{m}$ be a prime ideal of $\mathcal{O}_F.$ The algorithm outputs a set $C_{m-1}$ satisfying the following conditions:\\
		(i) The set $C_{m-1}$ consists of the  representatives of some elements in $(k_{v_{m}})^{*}$;\\
		(ii) For any element $c\in{C_{m-1}},$ the equation $N(c)=min\{N(c+t\alpha_{m})|t\in{\mathcal{O}_F}\}$  holds.\\
		\STATE\textbf{1.[Initialize]} Set num\_good$\leftarrow0$ and $C_{m-1}\leftarrow{\emptyset}$. Invoking the methods \textbf{GEN getc\_1()}
		and \textbf{GEN getc\_2()}, we can get the constant numbers $c_1$ and  $c_2,$ respectively.
		Moreover, invoking the PARI's function \textbf{GEN ideallist0(GEN nf, long bound, long flag)} and \textbf{long pr\_get\_f(GEN pr)},
		we can get the residue class degree $f_m$ of $v_m$ and the set $C'_{m-1}$ of all ideals whose norms are less than or equal to $(c_1c_2)^2N(v_m),$
		respectively. Invoking the method \textbf{GEN Cideal::getSetInitG()},
		we can get the unique  element $g\in{G_{m-1}}.$
		Lastly, set num\_C$'\leftarrow|C'_{m-1}|$ and num\_C$\leftarrow{N(V_m)-1}.$ \\
		\STATE\textbf{2.[Compare num\_good with num\_C]} If num\_good=num\_C  holds, the algorithm return $C_{m-1}$ and is terminated.\\
		\STATE\textbf{3.[Set the germs of $C_{m-1}$]} For $1\leq{i}\leq{num\_C}$, if $f_m=1$ set c\_i$\leftarrow{i}$.
		Otherwise, set c\_i$\leftarrow{g^i}.$\\
		\STATE\textbf{4.[Look for the appropriate elements in $C_{m-1}$]}
		For $1\leq{j}\leq{num\_C'}$ and $1\leq{k}\leq{70},$  set $c'\_j\leftarrow C'_{m-1}[j].$
		Then invoke the \textbf{PARI}'s function \textbf{long idealval(GEN nf, GEN x, GEN pr)}
		to decide whether or not $c_i-c'_j\xi^k$ is in $U_m.$
		Let the function's returned value be $b$.
		If $b>0$, set $c_i\leftarrow{c'_j\xi^k}$ and $num\_good\leftarrow{num\_good+1},$ and go to step 2;
		otherwise, set $j\leftarrow j+1$ and $k\leftarrow k+1.$
	\end{algorithmic}
\end{algorithm}
\

\

\

\

\

\

\

\

\

\

\

\

\

\

\

\

\

\

\

\

\
In the class {\it Cideal}, the methods \textbf{GEN Para\_getSetInitC()} and \textbf{static void* Part\_getSetInitC(void *arg)}
are the parent thread and the child thread respectively.
Using these methods, we can obtain the set $C_{m-1}$ corresponding to the prime ideal $v_{m}.$
Moreover we can change the number of the child threads with  different computer's hardware.

In the class {\it Cideal}, the two methods introduced above are
\textbf{bool Cideal::checkC\\
	onditionOne()} and
\textbf{bool Cideal::checkConditionTwo()}.
As a result, conditions I and condition II can be verified respectively for the prime ideal $v_{m}$ by  the two methods,
which implement respectively Algorithm 4.14 and Algorithm 4.15 below.

\begin{algorithm}
	\caption{(The algorithm on {\bf checkConditionOne()})}
	\begin{algorithmic}
		\STATE Let $v_{m}$ be a prime ideal of $\mathcal{O}_F.$ The algorithm check whether or not condition I is hold for $v_m.$\\
		\STATE\textbf{1.[Initialize]} Set num\_good$\leftarrow0$. Invoking the methods \textbf{GEN getSetInitW()},
		\textbf{GEN getSetInitC()}and \textbf{getUm()}, we can get the sets $W_{m-1},C_{m-1}$ and $U_m,$ respectively.
		Moreover, we can get the cardinal numbers of the sets $W_{m-1},C_{m-1},$
		denoted  by $num\_W$ and $num\_C,$ respectively.\\
		\STATE\textbf{2.[Compare num\_good with num\_W]} If num\_good=num\_W  holds, the algorithm returns true and is terminated.\\
		\STATE\textbf{3.[Loop through all element in SetW]} For $1\leq{i}\leq{num\_W}$, set w\_i$\leftarrow{W_{m-1}[i]}$.\\
		\STATE\textbf{4.[Look for the appropriate element $c$ in setC]}
		For $1\leq{j}\leq{num\_C},$ set c\_j$\leftarrow{C_{m-1}[j]}$.
		Invoking the \textbf{PARI}'s function \textbf{GEN bnfissunit(GEN bnf, GEN sfu, GEN x)} and getting it's returned value $b$,
		we can decide whether or not $\frac{w_i}{c_j}-1$ is in $U_m.$ More precisely,   if $b>0$, set $num\_good\leftarrow{num\_good+1}$ and go to step 2;
		otherwise, set j$\leftarrow j+1.$
	\end{algorithmic}
\end{algorithm}
\begin{algorithm}
	\caption{(The algorithm on {\bf checkConditionTwo()})}
	\begin{algorithmic}
		\STATE Let $v_{m}$ be a prime ideal of $\mathcal{O}_F.$ The algorithm check whether or not condition II  holds for $v_m.$\\
		\STATE\textbf{1.[Initialize]} Set num\_good$\leftarrow0$. Invoking the methods \textbf{GEN getSetInitG()},
		\STATE\textbf{GEN getSetInitC()}and \textbf{getUm()}, we can get the sets $G_{m-1}=\{g\},C_{m-1}$ and $U_m,$ respectively.
		Moreover, we can get the cardinal number of the set $C_{m-1}.$
		denoted  by $num\_C.$\\
		\STATE\textbf{2.[Compare num\_good with num\_C]} If num\_good=num\_C  holds, the algorithm returns true and is terminated.\\
		\STATE\textbf{3.[Loop through all element in SetC]} For $1\leq{i}\leq{num\_C}$, set c\_i$\leftarrow{C_{m-1}[i]}$.\\
		\STATE\textbf{4.[Look for the appropriate element $c'$ in setC]}
		For $1\leq{j}\leq{num\_C},$ set c'\_j$\leftarrow{C_{m-1}[j]}$.
		Invoking the \textbf{PARI}'s function \textbf{GEN bnfissunit(GEN bnf, GEN sfu, GEN x)} and getting it's returned value $b$,
		we can decide whether or not $\frac{c_i}{gc'_j}-1$ is in $U_m$  where $g$ is the unique element in set $G_{m-1}.$
		More precisely,  if $b>0$, set $num\_good\leftarrow{num\_good+1}$ and go to step 2;
		otherwise, set j$\leftarrow j+1.$
	\end{algorithmic}
\end{algorithm}

\end{subsubsection}
\end{subsection}
\end{section}
\

\

\

\

\

\

\

\

\

\

\

\begin{section}{The proof of Theorem 1.2}

Let $F=\mathbb{Q}\Big(\sqrt{-(D+B\sqrt{D})}\Big)$ be an imaginary cyclic quartic field. For the case  $B=1,D=2,$ invoking the method \textbf{GEN CquarField::getBoundOne()}(resp.\textbf{GEN CquarField::getBoundOne()})
we can know that for the prime ideals whose norms are greater than or equal to $172.525$ (resp. $3253.529$), condition I (resp. condition II) holds.
Moreover, by invoking the method \textbf{bool Ccheck::Para\_checkCon\\dition
	One(int num\_thread)}
(resp. \textbf{bool Ccheck::Par
	a\_checkConditionOne (int num\_thread)}), it is proved that for the prime ideals whose norms are less than
$172.525$ (resp. $3253.529$), condition I (resp.condition II) holds also.

Similarly, for the case  $B=2,D=13$ we can show that\\
(i) the bound determined by Lemma 3.4 ( resp. Theorem 3.6) is $1173.7$ (resp. $17321.1$);\\
(ii) for the prime ideals whose norms are less than $1173.7$(resp. $17321.7$), condition I (resp.condition II) holds.

And for the case  $B=2,D=29$ we can show that\\
(i) the bound determined by Lemma 3.4 ( resp. Theorem 3.6) is $48710.1$ (resp. $192289.6$);\\
(ii) for the prime ideals whose norms are less than $48710.1$(resp. $192289.6$), condition I (resp.condition II) holds.
For $F=\mathbb{Q}\Big(\sqrt{-(2+\sqrt{2})}\Big),$ $\mathbb{Q}\Big(\sqrt{-(13+2\sqrt{13})}\Big)$ or $\mathbb{Q}\Big(\sqrt{-(29+2\sqrt{29})}\Big)$ by PARI/GP, we know  that the torsion element is only  $-1.$
Hence, it is easy to show that $K_2\mathcal{O}_F$ can be generated by the two elements of order 2:
$\{-1,-1\}, \{-1,\xi\},$
where $\xi$ is a fundamental  unit of $F.$

However, in \cite{Browkin000},  Browkin proved the following formula:
$$\mbox{2-rank}K_2\mathcal{O}_F=r_1(F)+g(2)-1+\mbox{2-rank}\Big(\mbox{Cl}(F)/\mbox{Cl}_2(F)\Big),$$
where $r_1(F)$ is the number of real places of $F, g(2)$
the number of primes over $2$,
and Cl$(F)$  the class group of $F.$
It is well known that Cl$(F)=1$ and that by $PARI/GP,$ there is only one prime in $\mathcal{O}_F$ lying over $2.$
So  the formula takes the form:
$$\mbox{2-rank}K_2\mathcal{O}_F=0+1-1+0=0.$$
Hence there is no element of order 2. Thus the tame kernel $K_2\mathcal{O}_F$ is trivial. The proof is completed.
\begin{remark}\quad
	For $F=\mathbb{Q}\Big(\sqrt{-(13+2\sqrt{13})}\Big)$ and $F=\mathbb{Q}\Big(\sqrt{-(29+2\sqrt{29})}\Big)$, we keep a record of every $C_{m-1}$ in some text files, which can be found in
	\url{http://pan.baidu.com/s/1kVnSOCn} and
	\url{https://pan.baidu.com/s/1dFRn8ch} respectively.
\end{remark}
\end{section}
\bibliographystyle{amsplain}

\end{document}